\documentclass[reqno, 11pt]{article}

\usepackage{graphicx}%
\usepackage{multirow}%
\usepackage{amsmath,amssymb,amsfonts}%
\usepackage{amsthm}%
\usepackage{mathrsfs}%
\usepackage[title]{appendix}%
\usepackage{xcolor}%
\usepackage{textcomp}%
\usepackage{manyfoot}%
\usepackage{booktabs}%
\usepackage{algorithm}%
\usepackage{algorithmicx}%
\usepackage{algpseudocode}%
\usepackage{listings}%

\usepackage{fullpage}

\usepackage{enumerate}
\usepackage{comment}
\usepackage{newlfont}
\usepackage{color}
\usepackage{hyperref}

\usepackage{stmaryrd}
\SetSymbolFont{stmry}{bold}{U}{stmry}{m}{n}

\usepackage{bbm}

\usepackage{stmaryrd}
\usepackage{mathrsfs}

\usepackage{fancyhdr}
\usepackage{graphicx}
\usepackage{fancybox}
\usepackage{setspace}
\usepackage{cleveref}

\newtheorem{thm}{Theorem}

\newtheorem{lem}[thm]{Lemma}
\newtheorem{prop}[thm]{Proposition}

\theoremstyle{definition}
\newtheorem{defn}[thm]{Definition}

\theoremstyle{remark}
\newtheorem{rem}[thm]{Remark}

\newcommand{\R} {\mathbb{R}}
\newcommand{\C} {\mathbb{C}}

\newcommand{\E} {\mathbb{E}}

\DeclareMathOperator{\diag}{diag}

\DeclareMathOperator{\Tr}{Tr}

\DeclareMathOperator{\re}{\mathrm{Re}}
\DeclareMathOperator{\im}{\mathrm{Im}}

\newcommand{\beq}{ \begin{equation} }
\newcommand{\eeq}{ \end{equation} }

\newcommand{\dd}{\mathrm{d}}

\renewcommand{\P}{\mathbb{P}}

\numberwithin{equation}{section} 
\numberwithin{thm}{section}

\begin{document}
\include{amsthm_sc}

\title{Fluctuations of the free energy of the spherical Sherrington--Kirkpatrick model with heavy-tailed interaction}


\author{Taegyun Kim\footnote{Department of Mathematical Sciences, KAIST, Daejeon, 305701, Korea \newline email: \texttt{ktg11k@kaist.ac.kr}}
and Ji Oon Lee\footnote{Department of Mathematical Sciences, KAIST, Daejeon, 305701, Korea
\newline email: \texttt{jioon.lee@kaist.edu}}}



\maketitle

\abstract{
We consider the 2-spin spherical Sherrington--Kirkpatrick model without external magnetic field where the interactions between the spins are given as random variables with heavy-tailed distribution. We show that the free energy exhibits a sharp phase transition depending on the location of the largest eigenvalue of the interaction matrix. We also prove the order of the limiting free energy and the limiting distribution of the fluctuation of the free energy for both regimes.

\textbf{keywords}:{Heavy-tailed matrix, SSK model, Free energy, Phase transition}}
\section{Introduction}
The spin glass model, initially introduced in the seminal work of Edwards and Anderson \cite{edwards1975theory}, has evolved significantly over time. One notable advancement was the development of its mean-field variant, the Sherrington--Kirkpatrick (SK) model \cite{SKmodel}. Despite its apparent simplicity, the SK model captures many fundamental properties of the real-world spin glass model, which has led to its widespread study and application in a variety of contexts. A critical aspect of the SK model is its free energy, which can be conceptualized through the replica symmetry breaking method, notably the Parisi formula \cite{Parisi}, which was later mathematically validated \cite{guerra2003broken,TalagrandParisi}.

A notable variant of the SK model is the spherical SK model, introduced by Kosterlitz, Thouless, and Jones \cite{KosterlitzThoulessJones}. Their model replaces the standard spin vectors with those uniformly distributed on a sphere. The spherical SK model has garnered attention for its applications in statistical mechanics and other fields such as information theory \cite{mezard2009information}, with significant progress made in understanding its mathematical properties, including its limiting free energy \cite{TalagrandParisiSpher} and the fluctuation of the free energy \cite{2016fluctuations}. For other variants of the SSK model, we refer to \cite{baik2021spherical} and references therein.

Although the exact formulas for the limit and the fluctuation of the free energy may depend on the details of the model, we can find one common feature in these models: The limiting free energy is a deterministic function of the inverse temperature. Technically, it can be understood as follows: (1) by the integral representation formula (e.g., Lemma 1.3 in \cite{2016fluctuations}), the partition function depends only on the eigenvalues of the interaction matrix, (2) the eigenvalues are close to the deterministic locations determined by their limiting empirical spectral distribution (ESD), (3) as a result, the limiting free energy can be well-approximated by a deterministic function of the inverse temperature.

In this paper, we consider a variant of the SSK model without external magnetic field in which the interactions between the spins are given as random variables with heavy-tailed distribution and investigate how this heavy-tailed interaction can alter the continuity of the limiting free energy. The non-spherical counterpart of such a model, i.e., the SK model with heavy-tailed interaction, was considered in \cite{cizeau1993mean,cizeau1994theory} and its properties including the phase transitions of several thermodynamic quantities were analyzed by replica method in \cite{janzen2008replica,janzen2010thermodynamics,neri2010phase}. Recently, the existence of free energy and its fluctuation in the high-temperature regime for the heavy-tailed SK model was proved in \cite{jagannath2022existence,chen2023rigorous} and its disordered chaos phenomenon at zero temperature was shown in \cite{chen2024disorder}. These results reveal the significant differences between the heavy-tailed SK model and the usual SK model. While there might be several reasons that the heavy-tailed SK or SSK model is different from other models, we can consider the following: (1) the ESDs of the heavy-tailed random matrices are markedly different from corresponding `light-tailed' Wigner random matrices, and (2) the conventional techniques such as the integration by parts cannot be used due to the unboundedness of the first a few moments (most importantly, the variance), unlike the analysis for the usual SK model, e.g., the proof of the existence of the thermodynamic limit for the free energy in \cite{guerra2002thermodynamic}.

Before describing the precise detail of our model, we recall the key results for the usual SSK model \cite{2016fluctuations}. The Hamiltonian for the SSK model is defined by
\beq
H_N(\sigma)=-\frac{1}{\sqrt{N}}\sum_{i,j=1}^N H_{ij}\sigma_i\sigma_j=-\frac{1}{\sqrt{N}}\langle\sigma, H\sigma\rangle
\eeq
for a real symmetric random matrix $H$, where $\sigma = (\sigma_1, \sigma_2, \dots, \sigma_N) \in S_N := \{ x \in \R^N : x_1^2 + \dots + x_N^2 = N \}$. The entries of $H$ are with mean $0$ and variance $1$, so the largest eigenvalue of $H$ is of order $\sqrt{N}$ with high probability. The partition function and its free energy of the SSK model is defined as
\beq
F_N \equiv F_N(\beta)=\frac{1}{N} \log{Z_N}, \quad Z_N=\int_{S_N} e^{\beta H_n(\sigma)} \dd\sigma,
\eeq
where $\beta$ is the inverse temperature and $\dd \sigma$ is the Haar measure (normalized uniform measure) on $S_N$. It is known that
\beq \label{eq:F_N_SSK}
	F_N(\beta) \to F(\beta) \equiv
	\begin{cases}
	\beta^2 & \text{ if } 0 < \beta \leq 1/2, \\
	2\beta - \frac{\log(2\beta) + \frac{3}{2}}{2} & \text{ if } \beta > 1/2,
	\end{cases}
\eeq
as $N \to \infty$. Moreover, the fluctuation of the free energy is given by Gaussian law of order $N^{-1}$ when $\beta < 1/2$ and GOE Tracy--Widom law of order $N^{-2/3}$ when $\beta > 1/2$.                                                                                      

In our model, we further assume that the entries $H$ are heavy-tailed random variables in the sense that $\P (|H_{ij}| > u) \sim u^{-\alpha}$ for some $\alpha \in (0, 2)$. The exponent $\alpha$ determines the heaviness of the tail of the distributions. 
With heavy-tailed entries, unlike the light-tailed model, the largest eigenvalue is not of the order $1$ but grows with $N$. We thus need to introduce an additional normalization factor $b_N$, which is asymptotically of order $N^{2/\alpha}$, and consider the partition function defined as
\beq \label{eq:Z_N}
Z_N=\int_{S_N} e^{b_N^{-1} \beta H_N(\sigma)} \dd\sigma.
\eeq
See Section \ref{subsec:def}, especially Definition \ref{partition function}, for the precise definition of the model and the partition function $Z_N$ with the normalization factor $b_N$. We note that the normalization factor $b_N$ ensures that the extreme eigenvalues of $b_N^{-1} H$ are of order $1$; see \cite{LargeLevy2009,soshnikov2004poisson} for more details.

\subsection{Main contribution}

In this paper, we prove the behavior of the (log-) partition function in the ``high temperature regime'' and the ``low temperature regime''. We find that
\begin{itemize}
\item (Theorem \ref{main}.(ii)) In the high temperature regime, $\log Z_N$ converges in distribution to a random variable, independent of $N$, as $N \to \infty$.
\item (Theorem \ref{main}.(iii)) In the low temperature regime, $N^{-1} \log Z_N$ converges in distribution to another random variable, independent of $N$, as $N \to \infty$.
\end{itemize}

When compared to the behavior of the free energy for the usual SSK model, we notice several contrasts. First, while for the usual SSK model the size of $\log Z_N$ is of order $N$ regardless of the temperature, it is markedly different for our model in the high temperature regime and the low temperature regime - the former is of order $1$ whereas the latter is of order $N$. Second, for the usual SSK model, the fluctuation of $\log Z_N$ is much smaller than its deterministic limit; they are of comparable order for our model.

The transition between the high temperature regime and the low temperature regime is even more striking for our model. For the usual SSK model, the behavior of the free energy is solely determined by the inverse temperature. It is in the high temperature regime if $\beta$ is below a certain threshold, $1/2$ in \eqref{eq:F_N_SSK}, and in the low temperature regime if $\beta$ is below the threshold; no randomness is involved in the phase transition. On the other hand, with heavy-tailed interaction matrix, the distinction between the two phases is more subtle. In Theorem \ref{main}, we prove that the high temperature regime is when the largest eigenvalue of the interaction matrix is above a certain threshold determined by $\beta$ and the low temperature regime is when the largest eigenvalue is below the threshold. In particular, it implies that the size of $\log Z_N$ is of order $1$ with probability $P_{\beta}$ and of order $N$ with probability $1-P_{\beta}$. The change in inverse temperature only changes probability $P_{\beta}$.

For the proof of the main result, we adapt the strategy introduced in \cite{2016fluctuations}. In this approach, the partition function is first written as a complex integral involving the eigenvalues of the interaction matrix (see \eqref{eq:integral}). Then, with the aid of several estimates on the eigenvalues of heavy-tailed matrices, the integral is approximated by applying the method of steepest descent. In the high temperature regime, the (log-) partition function can be approximated by a linear spectral statistics (LSS) of the eigenvalues, while it is governed by the largest eigenvalue in the low temperature regime.

The main technical difficulty in the analysis is due to that some spectral properties of heavy-tailed matrices are not known, e.g., its LSS with respect to a logarithmic function. To handle it, we notice that the fluctuation of several quantities, such as the trace and the Hilbert-Schmidt norm of a heavy-tailed matrix is governed by a few extreme entries of the matrix. Then, by applying well-known facts about heavy-tailed distributions, we can analyze the LSS with respect to a logarithmic function that is essential in the analysis for high temperature regime.

\subsection{Related works}

The SSK model was first introduced by Kosterlitz, Thouless, and Jones \cite{KosterlitzThoulessJones}. The limiting free energy for the SSK model was computed (nonrigorously) in \cite{KosterlitzThoulessJones} and proved by Talagrand \cite{TalagrandParisiSpher}. The fluctuation of the free energy for the SSK model was proved in \cite{2016fluctuations}. Corresponding results for several variants of the SSK model were proved in \cite{2016fluctuations,baik2017fluctuations,Baik-Lee2018,lee2023spherical,kim2023fluctuations}. Several other results were also shown, including the near-critical behavior \cite{BaikLeeWu,landon2022free,johnstone2021spin} and the overlap \cite{landon2022fluctuations}. For more results on the SSK model and its variants, we refer to \cite{baik2021spherical} and references therein.

The spectral properties of the heavy-tailed random matrix have been extensively studied in random matrix theory. The behavior of the largest eigenvalue was first proved by Soshnikov \cite{soshnikov2004poisson} for the case $\alpha < 2$ and later generalized to the case $\alpha < 4$ by Auffinger, Ben Arous, and P\'ech\'e \cite{LargeLevy2009}. Several other results were proved, including the spectrum normalization \cite{levyspectrum2008}, a central limit theorem for the Stieltjes transform \cite{2014centralheavy}, and bulk universality \cite{aggarwal2021goe}.

\subsection{Organization of the paper}
In Section \ref{sec:result}, we define our model and state the main results. In Section \ref{sec:prelim}, we introduce several important properties for heavy-tailed random matrices. In Sections \ref{sec:high} and \ref{sec:low}, we prove the main results in the high temperature regime and the low temperature regime, respectively. Several technical details, including the proofs of the lemmas, can be found in Appendix.

\section{Definitions and main results} \label{sec:result}

In this section, we precisely define our model and state the main results.

\begin{rem}[Notational remark]
Throughout the paper, we use the standard asymptotic notations such as $O$, $o$, and $\Theta$ as $N \to \infty$. We denote by the symbol $\Rightarrow$ the convergence in distribution as $N \to \infty$. We use the shorthand notation $\sum_i = \sum_{i=1}^N$.
\end{rem}

\subsection{Definitions} \label{subsec:def}

We begin by defining the heavy-tailed random variables.

\begin{defn}[Heavy-tailed random variable] \label{stable law}
  We say that a random variable $X$ is heavy-tailed with exponent $\alpha$ if $\P(|X|>u)=L(u){u^{-\alpha}}$ for some $\alpha \in (0, 2)$ and some slowly varying function $L$, i.e. $\lim_{x\to\infty} L(tx)/L(x)=1$ for all $t$. We further assume that for any $\delta > 0$, there exists $x_0$ such that $L(x) e^{x^{\delta}}$ is an increasing function on $(x_0, \infty)$. 
\end{defn}

\begin{rem}\label{rem:slow_vary}
In Definition \ref{stable law}, the assumption on the monotonicity of the function $L(x) e^{x^{\delta}}$ is purely technical, which guarantees that the difference $\lambda_1 - \lambda_2$ is bounded below with high probability (in Lemma \ref{lem:gap_eig}). This assumption is satisfied by many slowly varying functions considered in the literature, especially poly-log functions $(\log(x))^p$ for some $p \in \mathbb{R}$.
\end{rem}

Note that a heavy-tailed random variable with exponent $\alpha$ has a finite $(\alpha-\epsilon)$ moment for any $\epsilon \in (0, \alpha)$. We next define a Wigner matrix with heavy-tailed entries.

\begin{defn}[Heavy-tailed Wigner matrix] \label{stable matrix}
  We say that an $N \times N$ real symmetric matrix $M$ is a heavy-tailed Wigner matrix with exponent $\alpha$ if its upper triangular entries $M_{ij} (1 \leq i \leq j \leq N)$ are independent and identically distributed (i.i.d.) heavy-tailed random variables with exponent $\alpha$, defined in Definition \ref{stable law}.
\end{defn}

Our model is a spherical SK model where the interactions between spins are i.i.d. heavy-tailed random variables. The partition function and free energy are defined as follows:
\begin{defn}[Partition function] \label{partition function}
  For an $N \times N$ heavy-tailed Wigner matrix $M$ with exponent $\alpha$, let $b_N$ be a normalization factor defined by
\beq \label{eq:b_N}
	b_N = \inf \{t: \P(|M_{11}|>t)\leq \frac{2}{N(N+1)}\}.
\eeq
We remark that $b_N$ is approximately of order $N^{2/\alpha}$, neglecting the correction due to a slowly varying function.

We define the partition function associated with $M$ at inverse temperature $\beta$ by
\beq
	Z_N \equiv Z_N(\beta) = \int_{S_N} \exp(\beta b_N^{-1} \langle x, Mx \rangle) \, \dd \Omega_N(x),
\eeq
where $S_N := \{ x \in \R^N : x_1^2 + \dots + x_N^2 = N \}$ and we denote by $\dd \Omega_N$ the Haar measure (normalized uniform measure) on $S_N$.
\end{defn}

\subsection{Main results}

Our main result is the following theorem on the limiting behavior of the free energy.
\begin{thm}\label{main}
  Suppose that $M$ is an $N \times N$ heavy-tailed Wigner matrix with exponent $\alpha$ as in Definition \ref{stable matrix}. Then, there exist events $F_1$ and $F_2$ such that 
	\begin{enumerate}[(i)]
	\item Events $F_1$ and $F_2$ are mutually exclusive with
	\[
		\lim_{N \to\infty} \P(F_1)= \exp(-(2\beta)^{\alpha}), \quad \lim_{N \to\infty} \P(F_2)= 1-\exp(-(2\beta)^{\alpha})
	\]
	
	\item Given $F_1$, for the partition function $Z_N$ in Definition \ref{partition function}, as $N \to \infty$
	\[
		(\log Z_N)|F_1 \Rightarrow T
	\]
	for some ($N$-independent) random variable $T$.
	
	\item Given $F_2$, for the partition function $Z_N$ in Definition \ref{partition function}, as $N \to \infty$
	\[
		 (\frac{1}{N}\log Z_N)|F_2 \Rightarrow \beta X-\frac{1}{2}\log(2e\beta X)
	\]
	for some ($N$-independent) random variable $X$.
	\end{enumerate}	
\end{thm}

The precise definitions of $F_1$ and $F_2$ in Theorem \ref{main} are
\beq
	F_1 = \{ \lambda_1 < \frac{b_N}{2\beta} \}, \quad F_2 = \{ \lambda_1 > \frac{b_N}{2\beta} \}.
\eeq
As we discussed in Introduction, the events $F_1$ and $F_2$ correspond to the high temperature regime and the low temperature regime, respectively. We remark that $X$ in the low temperature regime is a random variable such that $X>1/(2\beta)$ and its distribution is given by
\[
	\P(X>u)=\frac{1-\exp(-u^{-\alpha})}{1-\exp(-(2\beta)^{\alpha})}.
\]
We also remark that the random variable $T$ in the high temperature regime has undefined mean.

\begin{proof}[Proof of Theorem \ref{main}]
See Proposition \ref{high_temp} in Section \ref{sec:high} and Proposition \ref{low_temp} in Section \ref{sec:low}. (See also Proposition \ref{Largest eigenvalue} for about the limits of $\P(F_1)$ and $\P(F_2)$.)
\end{proof}

\subsection{Strategy of the proof}

For the proof of Theorem \ref{main}, we adapt the integral representation formula for the partition function, introduced in \cite{KosterlitzThoulessJones} and proved in \cite{2016fluctuations}. Let $\lambda_1 \geq \lambda_2 \geq \dots \geq\lambda_N$ be the eigenvalues of $M$. In Lemma \ref{lem:integral representation}, we prove that
\beq \label{eq:integral}
 Z_N = C_N \frac{1}{i} \int_{\gamma-i\infty}^{\gamma+i\infty} e^{\frac{N}{2b_N}G(z)} \dd z
\eeq
for any $\gamma > \lambda_1$, where
\beq \label{eq:G}
  G(z)=2\beta z-\frac{b_N}{N}\sum_i \log(z-\lambda_i), \quad C_N=\frac{\Gamma(N/2)b_N^{N/2-1}}{2\pi (N\beta)^{N/2-1}}.
\eeq
To analyze the asymptotic behavior of the right side of \eqref{eq:integral}, in the viewpoint of the method of steepest descent, the best candidate for $\gamma$ would be the critical point of the function $G$, i.e., the solution of the equation
\beq \label{eq:gamma}
	G'(z) =2\beta-\frac{b_N}{N}\sum_i \frac{1}{z-\lambda_i} = 0.
\eeq
Note that $G'(z)$ is an increasing function of $z$ and \eqref{eq:gamma} has a unique solution in $(\lambda_1, \infty)$. 

If $\lambda_1<b_N/2\beta$, which we call the high temperature regime, since $\lambda_i \ll \lambda_1$ holds for all but $o(N)$ eigenvalues $\lambda_i$, it is possible to approximate $G'(\gamma)$ by $2\beta-b_N/\gamma$ with negligible error. We can then approximate the solution of equation \eqref{eq:gamma} using a deterministic number $b_N / (2\beta)$. Thus, we can analyze the asymptotic behavior of $Z_N$ by choosing $\gamma = b_N / (2\beta)$. In Section \ref{sec:high}, we consider this case and prove the result on the fluctuation of $\log Z_N$, Proposition \ref{high_temp}, by considering the linear spectral statistics of $M$ involving the function $G(z)$ in \eqref{eq:G}.

If $\lambda_1 > b_N / (2\beta)$, on the other hand, we cannot choose $\gamma$ to be $b_N / (2\beta)$. In this case, which we call the low temperature regime, we prove that the solution of the equation \eqref{eq:gamma} sticks to the largest eigenvalue $\lambda_1$, and the behavior of $Z_N$ is directly related to that of $\lambda_1$ with negligible error. In Section \ref{sec:low}, we prove the limiting behavior of $Z_N$ as in Proposition \ref{high_temp}, by considering the behavior of $\lambda_1$.

For both the high temperature regime and the low temperature regime, it is crucial in the analysis to understand the behavior of the eigenvalues of heavy-tailed Wigner matrices. We thus introduce several results on the eigenvalues of $M$ in the next section, before we analyze the integral in the right side of \eqref{eq:integral}.

\section{Preliminaries} \label{sec:prelim}

\subsection{Properties of heavy-tailed Wigner matrices}

In this section, we collect several important results on heavy-tailed random variables and heavy-tailed Wigner matrices, which will be used in the analysis of our model. (For more results on the eigenvalues of heavy-tailed Wigner matrices, we refer to \cite{soshnikov2004poisson,basrak2021extreme}.) For the sake of convenience, throughout this section we assume that $M$ is an heavy-tailed Wigner matrix with exponent $\alpha$ as in Definition \ref{stable matrix} and each entry $M_{ij}$ satisfies $\P(|M_{ij}|>u)=G(u)=L(u) u^{-\alpha}$, where $L$ is a slowly varying function as in Definition \ref{stable law}. We let $\lambda_1 \geq \dots \geq \lambda_N$ be the eigenvalues of $M$. We first define the high probability events.

\begin{defn}[High probability event]
  We say that an event $F_N$ holds with high probability if there exist (small) $\delta>0$ and (large) $N_0$ such that $\P(F_N)\geq 1-N^{-\delta}$ holds for any $N > N_0$.
\end{defn}

The distribution of the largest eigenvalue $\lambda_1$ of $M$ is given as follows:
\begin{prop} \label{Largest eigenvalue}
	The limiting distribution of the rescaled largest eigenvalue $\lambda_1$ is given by
	\[
		\lim_{N\to \infty}\P\left(\frac{\lambda_1}{b_N}\leq x\right)=\exp(-x^{-\alpha}),
  \]
  where the normalization factor $b_N$ is as given in \eqref{eq:b_N}.
\end{prop}
For the proof of Proposition \ref{Largest eigenvalue}, see Theorem 1 in \cite{soshnikov2004poisson}. Next, we state the results which ensure that the good events for $M$ hold with high probability.
\begin{lem}\label{lem:whp_set1}
  With high probability, the following hold:
    \begin{itemize}
        \item For any $1 \leq i \leq N$, $|M_{ii}| \leq b_N^{11/20}$  
        \item For any $1 \leq i \leq j \leq N$, we have $|M_{ij}| \leq b_N^{99/100}$ or $|M_{ii}|+|M_{jj}|\leq b_N^{1/10}$
        \item For any constant $\delta>0$, there is no row of $M$ that has at least two entries whose absolute values are larger than $b_N^{1/2+\delta}$.
    \end{itemize}
\end{lem}
See \cite{soshnikov2004poisson} for the proof. 

\begin{lem}\label{inf_norm}
  For every $\epsilon> 0$, with high probability the following holds for any $i=1, 2, \dots, N$: Suppose that $|M_{ii'}| = \max_{j} |M_{ij}|$. Then, $|M_{ii'}| < b_N^{1/2+\epsilon}$ or $|\sum_{j: j \neq i'} M_{ij}| < b_N^{1/2+\epsilon}$.
\end{lem}
For the proof, see Appendix \ref{app:proof}. As a simple corollary to Lemma \ref{inf_norm}, we can also obtain the following estimates on the extreme eigenvalues of $M$.

\begin{lem}\label{ls_eig}
	With high probability, we have
  \[
		\lambda_1= (1+O(b_N^{-1/8}))\max\{|M_{ij}|\}, \quad |\lambda_N|= (1+O(b_N^{-1/8}))\max\{|M_{ij}|\}
  \]
\end{lem}
See Appendix \ref{app:proof} for the proof.
We also can show that $\lambda_1$ is away from $\lambda_2$ and 0 with high probability. This will be used to show approximate the value of steepest descent method.
\begin{lem}\label{lem:gap_eig}
    For every $\epsilon>0$, the following hold with high probability.
    \begin{itemize}
        \item  $\lambda_1>b_N N^{-\epsilon}$.
        \item $\lambda_1-\lambda_2> b_N N^{-\epsilon}$.
        \item $\lambda_1 \notin ( \frac{b_N}{2\beta} (1-N^{-\epsilon}),  \frac{b_N}{2\beta} (1+N^{-\epsilon}))$.
    \end{itemize}

\end{lem}
See Appendix \ref{app:proof} for the proof, where the assumption for slowly varying functions in Definition \ref{stable law} plays a role.
The next result is about the sum of i.i.d. heavy-tailed random variables. 
\begin{prop}\label{stable_dist}
Suppose $X_1,X_2,\cdots$ are i.i.d heavy-tailed random variables with exponent $\alpha$ defined in Definition \ref{stable law} such that
\[
	\lim_{x\to \infty} \P(X_1>x)/ \P(|X_1|>x)=\theta\in [0,1]
\]
Let $S_N=X_1+\cdots+X_N$, $a_N=\inf\{x: \P(|X_1|>x)\leq N^{-1}\}$, and $c_N=N \E [X_1 \mathbbm{1}_{(|X_1|\leq a_N)}]$. Then, there exists a random variable $Y$ whose distribution is non-degenerate such that $(S_N-c_N)/a_N \Rightarrow Y$.
\end{prop}

For the proof, see Theorem 3.8.2 in \cite{durrett2019probability}.

\begin{rem}
In Proposition \ref{stable_dist}, we can actually prove that both $c_N/a_N$ and $S_N/a_N$ converge for $\alpha<1$. To see this, for $\alpha<1$, we consider
\[
	\E [|X_1| \mathbbm{1}_{|X_1|\leq a_N}]=\int_{0}^{a_N} x g(x) \, \dd x = \left.-x G(x) \right|_0^{a_N}+\int_{0}^{a_N} G(x) \, \dd x
\]
  where $g$ is the probability density function for $|X_1|$ and $G(x)=\P( |X_1| > x)$. From the definition of $a_N$, we find that $N a_N^{-\alpha}L(a_N) = 1+o(1)$. 
	The last term in the right side of the equation above is then
\[
	\int_{0}^{a_N} G(x)\, \dd x=\int_{0}^{a_N} x^{-\alpha} L(x)\, \dd x =\frac{1}{1-\alpha}a_N^{-\alpha+1}L(a_N)(1+o(1))=\frac{1}{1-\alpha}\frac{a_N}{N}(1+o(1))
\]
by Karamata's theorem. (See, e.g., Thm 8.9.2 of \cite{feller1991introduction}). This implies
\[
	N\E[|X_1| \mathbbm{1}_{|X_1|\leq a_N}]/a_N\to \frac{1}{1-\alpha}-1.
\]
Further, from the assumption on $X_1$, we can find that
\[
	N\E[X_1 \mathbbm{1}_{0\leq X_1\leq a_N}]/a_N\to \theta (\frac{1}{1-\alpha}-1).
\]
We then conclude $$c_N/a_N\to (2\theta-1)(\frac{1}{1-\alpha}-1),$$
which shows that $c_N/a_N$ converges to a constant depending only on the probability density function. Together with Proposition \ref{stable_dist}, this also shows that $S_N/a_N$ converges.
\end{rem}

We can also prove the following bound for $a_N$ in Proposition \ref{stable_dist}, which will be used to prove an asymptotic behavior of the eigenvalue statistics.
\begin{lem} \label{lem:a_N}
Suppose that $X$ is a heavy-tailed random variables with exponent $\alpha$. Let $a_N=\inf\{x: \P(|X|>x)\leq N^{-1}\}$. Then, for any $\epsilon>0$, the following holds for any sufficiently large $N$:
\[
  N^{1/\alpha-\epsilon}\leq a_N\leq N^{1/\alpha+\epsilon}.
\]
In particular, for any $\epsilon>0$,
\[
	\lim_{N \to \infty} a_N/N^{1/\alpha+\epsilon}=0
\]
\end{lem}
See Appendix \ref{app:proof} for the proof.
\\In the following lemmas, we prove the asymptotic behavior of $\sum \lambda_i^2$. We begin by considering the following estimate for the moments of $X$.
\begin{lem} \label{lem:X^delta}
   Suppose that the assumptions in Lemma \ref{lem:a_N} holds. Assume further that $X$ is non-negative. Then, for any $0 < \delta < \alpha$, we have
	\[
		\E(X^{\delta})= O(1)
	\]
\end{lem}
See Appendix \ref{app:proof} for the proof. As a corollary, we can prove the following lemma.
\begin{lem}\label{lem:sum X^delta}
For non-negative heavy-tailed random variables $X_1,...,X_N$ with exponent $\alpha$,
$$\E((\sum_{i=1}^N X_i)^{\delta})= O(N)$$ for $\delta<\min\{\alpha, 1\}$.
\end{lem}
\begin{proof}
Jensen's inequality implies $\E((\sum_{i=1}^N X_i)^{\delta})\leq \sum_{i=1}^N \E(X_i^{\delta})=O(N).$
\end{proof}
We next show that $\sum_i \lambda_i^2<N^{4/\alpha+\epsilon}$ holds with high probability.
\begin{lem}\label{lem:sum_eig_sq_bound}
    For eigenvalues $\lambda_1,\cdots,\lambda_N$ of $M$,
    $$\P(\sum_{i} \lambda_i^2> N^{4/\alpha+\epsilon})= O(N^{-\alpha\epsilon/4}).$$
    In particular, $$\sum_{i} \lambda_i^2 <N^{4/\alpha+\epsilon}$$
		with high probability.
\end{lem}
See Appendix \ref{app:proof} for the proof.
The next lemma provides a bound for the number of the eigenvalues whose sizes are comparable to that of the largest eigenvalue.
\begin{lem}\label{lem:num_eig} For the eigenvalues $\lambda_1,\cdots,\lambda_N$ of $M$,
    $$\#\{ |\lambda_i|>b_N N^{-\epsilon}\}=O(N^{3\epsilon})$$
with high probability.
\end{lem}
See Appendix \ref{app:proof} for the proof.
Furthermore, we can prove that $\sum \lambda_i^2/b_N^2$ converges in distribution to a random variable, which will be used to show that our free energy also converges in distribution to a random variable.
\begin{lem}\label{lem:trace_square}
  The eigenvalues $\lambda_1,..,\lambda_N$ of M satisfy $$\sum_{i} \lambda_i^2/ b_N^2 \Rightarrow X$$ for some non-degenerate random variable $X$.
\end{lem}
See Appendix \ref{app:proof} for the proof.

In the following lemma, we show that the limit of $\sum_i \lambda_i/b_N$ is $0$.
\begin{lem}\label{lem:lim_of_Tr}
     The eigenvalues $\lambda_1,..,\lambda_N$ of $M$ satisfies $$\lim_{N\to\infty}\sum_{i=1}^N \lambda_i/ b_N =0.$$ 
\end{lem}
See Appendix \ref{app:proof} for the proof.

\subsection{Integral representation formula for the partition function}

\begin{lem} \label{lem:integral representation}
Let $Z_N$ be the partition function defined in Definition \ref{partition function}. Then,
\beq
 Z_N = C_N \frac{1}{i} \int_{\gamma-i\infty}^{\gamma+i\infty} e^{\frac{N}{2b_N}G(z)} \dd z
\eeq
for any $\gamma > \lambda_1$, where
\beq
  G(z)=2\beta z-\frac{b_N}{N}\sum_i \log(z-\lambda_i), \quad C_N=\frac{\Gamma(N/2)b_N^{N/2-1}}{2\pi (N\beta)^{N/2-1}}.
\eeq
\end{lem}
The proof of Lemma \ref{lem:integral representation} is standard (see, e.g., \cite{2016fluctuations}). For the detailed proof, see Appendix \ref{app:proof}.

Since $G'$ is an increasing function on the interval $(\lambda_1,\infty)$ with 
\[
   \lim_{z\to \lambda_1^+}G'(z)=-\infty, \quad \lim_{z\to \infty}G'(z)=2\beta>0,
\]
it is immediate to see that there exists a unique solution to the equation $G'(z)=0$, which with abuse of notation we call $\gamma$ in this section. In the following lemma, we prove a result on the relation between the critical point $\gamma$ and $\lambda_1$.

\begin{lem} \label{lem:gamma}
Set $\gamma \in (\lambda_1,\infty)$ as the (unique) solution to equation $G'(z) = 0$. Then,
\begin{itemize}
    \item If $\lambda_1<\frac{b_N}{2\beta}$, there exists a random variable $X_N$ such that 
		\[	
			\gamma=\frac{b_N}{2\beta}(1+\frac{X_N}{N})
		\]
		and $X_N$ converges in distribution to a non-degenerate random variable $X$ as $N \to \infty$.
    \item If $\lambda_1>\frac{b_N}{2\beta}$, for any $0<\epsilon<1/8$, 
		\[
			\gamma= \lambda_1+\frac{1}{2\beta-b_N/\lambda_1}\frac{b_N}{N}+O(\frac{b_N}{N^{1+\epsilon}})
		\]
		with high probability.
\end{itemize}

\end{lem}
For the proof, we can simply approximate $\gamma$ appropriately with some steps. See Appendix \ref{app:proof} for the proof.

\section{High temperature regime} \label{sec:high}
In this section, we consider the regime $\lambda_1<b_N/2\beta$. Recall that we let $\gamma=\frac{b_N}{2\beta}(1+\frac{X_N}{N})$ in Lemma \ref{lem:gamma}.
\begin{prop}\label{high_temp}
    For event $F_1=\left\{\lambda_1< \frac{b_N}{2\beta}\right\}$, 
    \[
        (\log Z_N)|F_1\Rightarrow T
    \]
    for some random variable $T$. Moreover, $\P(F_1) \to \exp(-(2\beta)^{\alpha})$ as $N \to \infty$.
\end{prop}

Throughout this section, we will assume that $|M_{ij}|<\frac{b_N}{2\beta}(1-N^{-\epsilon})$ for every $i,j$. To check the validity of this assumption, we first notice from the third part of Lemma \ref{lem:gap_eig} that for any fixed $\epsilon>0$, we have $\lambda_1 \notin (\frac{b_N}{2\beta}(1-N^{-\epsilon})), \frac{b_N}{2\beta}(1+N^{-\epsilon}))$ with high probability. Since $\lambda_1<\frac{b_N}{2\beta}$, it implies that $\lambda_1< \frac{b_N}{2\beta}(1-N^{-\epsilon})$ with high probability. We also note from Lemma \ref{ls_eig} that $\lambda_1=\max\{|M_{ij}|\}(1+O(b_N^{-1/8}))$ and $\lambda_N=-\max\{|M_{ij}|\}(1+O(b_N^{-1/8}))$ with high probability. Thus, given $F_1$, we find that $\max\{|M_{ij}|\}<\frac{b_N}{2\beta}(1-N^{-\epsilon})$ with high probability. 

Our first result in this section is a lemma that provides an approximation for the $k$-th derivative of $G$.
\begin{lem}\label{lem:high_tay}
   Fix $\delta > 0$. For any $|\gamma-a|=O(b_N N^{-\delta})$, the $k$-th derivative of $G$ satisfies $$G^{(k)}(a)=O(b_N^{-k+1})$$
	with high probability.
\end{lem}

\begin{proof}
From Lemma \ref{lem:num_eig},
    \[
        \#\{ |\lambda_i|>b_N N^{-\epsilon}\}=O(N^{3\epsilon}) 
    \]
    with high probability. Moreover, if we assume that $\lambda_1<\frac{b_N}{2\beta}$, then from Lemma \ref{lem:gap_eig} we find that for any fixed $\epsilon>0$, $\lambda_1<\frac{b_N}{2\beta}(1-N^{-\epsilon})$ with high probability.
    Since $\gamma=\frac{b_N}{2\beta}(1+\frac{X_N}{N})$ from Lemma \ref{lem:gamma}, we conclude that
    \[ \begin{split}
    G^{(k)}(a) &= \frac{b_N}{N}(k-1)!\sum_{i=1}^N \frac{1}{(a-\lambda_i)^k}= \frac{b_N(k-1)!}{N}\left(O(N^{3\epsilon})\frac{N^{k\epsilon}}{b_N^k}+(N-O(N^{3\epsilon}))O({b_N^{-k}})\right)\\
		&=O(b_N^{-k+1})
    \end{split} \]
    with high probability.
\end{proof}

\begin{proof}[Proof of Proposition \ref{high_temp}]
Recall from the integral representation formula for the partition function, Lemma \ref{lem:integral representation},
\begin{align*}
    \log Z_N=\log C_N+ \frac{N}{2b_N}G(\gamma)+\log\int_{-\infty}^{+\infty} e^{\frac{N}{2b_N}(G(\gamma+ti)-G(\gamma))}\dd t.
\end{align*}
We first handle the last term of the equation above. For a sufficiently small $\epsilon>0$,
\begin{align}
    \int_{-\infty}^{+\infty}& \exp\left({\frac{N}{2b_N}(G(\gamma+ti)-G(\gamma))}\right)\dd t=\sqrt{\frac{b_N}{N}}\int_{-\infty}^{+\infty}\exp\left({\frac{N}{2b_N}\left(G\left(\gamma+i\sqrt{\frac{b_N}{N}}t\right)-G(\gamma)\right)}\right)\dd t\nonumber
    \\&=\sqrt{\frac{b_N}{N}}\int_{|t|\leq\sqrt{b_N}N^{\epsilon}}\exp\left({\frac{N}{2b_N}\left(G\left(\gamma+i\sqrt{\frac{b_N}{N}}t\right)-G(\gamma)\right)}\right)\dd t\label{eq:small_high}
    \\&~~+\sqrt{\frac{b_N}{N}}\int_{|t|>\sqrt{b_N}N^{\epsilon}}\exp\left({\frac{N}{2b_N}\left(G\left(\gamma+i\sqrt{\frac{b_N}{N}}t\right)-G(\gamma)\right)}\right)\dd t\label{eq:large_high}.
\end{align}
For \eqref{eq:small_high}, by considering the Taylor expansion of $G$ and applying Lemma \ref{lem:high_tay},
\beq
\begin{split}
    &\int_{|t|\leq\sqrt{b_N}N^{\epsilon}}\exp\left({\frac{N}{2b_N}\left(G\left(\gamma+i\sqrt{\frac{b_N}{N}}t\right)-G(\gamma)\right)}\right)\dd t
   \\&=\int_{|t|\leq\sqrt{b_N}N^{\epsilon}}\exp\left({-
    \frac{1}{4}G''(\gamma)t^2+O(\frac{N}{2b_N}\frac{1}{3!}\sup _{|a-\gamma|<b_N N^{\epsilon-1/2}}|G^{(3)}(a)|(\sqrt{b_N}t/\sqrt{N})^3}\right)\dd t
    \\&=\int_{|t|\leq\sqrt{b_N}N^{\epsilon}}\exp\left({-
    \frac{1}{4}G''(\gamma)t^2+\frac{N}{2b_N} O(b_N^{-2}(b_N N^{\epsilon-1/2})^3)}\right)\dd t
    \\&=\int_{|t|\leq\sqrt{b_N}N^{\epsilon}}\exp\left({-
    \frac{1}{4}G''(\gamma)t^2+O(N^{-\epsilon})}\right)\dd t
    \\&=\int_{|t|\leq\sqrt{b_N}N^{\epsilon}}e^{(-
    \frac{1}{4}G''(\gamma))t^2}\dd t(1+O(N^{-\epsilon})).
\end{split}
\eeq
Since $G''(\gamma)=O(b_N^{-1})$, by applying the identity
\[
    \int_{x\geq a}e^{-\lambda x^2}\dd x = \int_{0}^{\infty}e^{-\lambda(t+a)^2}\dd t=e^{-\lambda a^2}\int_{0}^{\infty} e^{-\lambda(2ta+t^2)}\dd t\leq e^{-\lambda a^2}\int_{0}^{\infty} e^{-\lambda t^2}\dd t=\sqrt{\frac{\pi}{2\lambda}}e^{-\lambda a^2},
\] 
we obtain
\begin{align*}
    \int_{-\sqrt{b_N}N^{\epsilon}}^{\sqrt{b_N}N^{\epsilon}}e^{-
    \frac{1}{4}G''(\gamma)t^2}\dd t=\int_{-\infty}^{+\infty}e^{-
    \frac{1}{4}G''(\gamma)t^2}\dd t+O(e^{-N^{\epsilon}})=2\sqrt{\frac{\pi}{G''(\gamma)}}+O(e^{-N^{\epsilon}}).
\end{align*}
To estimate \eqref{eq:large_high}, we consider
\begin{align}
    &\left|\int_{|t|>\sqrt{b_N}N^{\epsilon}}\exp\left({\frac{N}{2b_N}\left(G\left(\gamma+i\sqrt{\frac{b_N}{N}}t\right)-G(\gamma)\right)}\right)\dd t\right|\nonumber
    \\&=\left|\int_{t>\sqrt{b_N}N^{\epsilon}}2\re \exp\left({\frac{N}{2b_N}(2\beta ti \sqrt{\frac{b_N}{N}}-\frac{b_N}{N}\sum_{i=1}^N\log(1-\frac{\sqrt{\frac{b_N}{N}}ti}{\gamma-\lambda_i}))}\right)\dd t\right| \nonumber
     \\&\leq\int_{t>\sqrt{b_N}N^{\epsilon}}2\left| \exp\left({\frac{N}{2b_N}(2\beta ti \sqrt{\frac{b_N}{N}}-\frac{b_N}{N}\sum_{i=1}^N\log(1-\frac{\sqrt{\frac{b_N}{N}}ti}{\gamma-\lambda_i}))}\right)\right|\dd t \nonumber
     \\&= \int_{t>\sqrt{b_N}N^{\epsilon}}2 \exp\left({-\frac{1}{4}\sum_{i=1}^N\log(1+\frac{{\frac{b_N}{N}}t^2}{(\gamma-\lambda_i)^2})}\right)\dd t.\label{eq:check}
\end{align}
Note that there are at least $N/2$ eigenvalues less than $\frac{b_N}{10\beta}$ with high probability, since $\sum \lambda_i^2<b_N^2 N^{\epsilon}$ with high probability. Using this fact, we further have
\begin{align*}
     |\eqref{eq:check}|
     &\leq \int_{t>\sqrt{b_N}N^{\epsilon}}2 \exp\left({-\frac{N}{8} \log(1+\frac{Ct^2}{b_N N})}\right)\dd t
     \\&\leq \int_{\sqrt{b_N}N^{\epsilon}}^{b_N N}\exp\left({-\frac{b_N C' N^{2\epsilon}}{8N}}\right)\dd t+\int_{t>b_N N}2\left(\frac{Ct^2}{b_N N}\right)^{-N/8}\dd t
     \\&= O(b_N N e^{-c'N^{2\epsilon}})+\left.\frac{1}{1-N/4}t^{-N/4+1}(c^2b_N N)^{N/8}\right|_{t=b_N N}^{t=\infty}
     \\&=O(b_N N e^{-c'N^{2\epsilon}})+O(N^{-N/8})
\end{align*}
for some $N$-independent constants $C,C',c,c'$. Here, to obtain the inequality in the second line, we used $\log(1+x)\geq x/2$ for $x<1$  for the first part and $\exp(-a\log (1+b))\leq b^{-a}$ for $a,b>0$ for the second part. 

Collecting the inequalities we obtained so far, we find
\begin{align*}
    \log \int_{-\infty}^{\infty}e^{\frac{N}{2b_N}(G(\gamma+ti)-G(\gamma))}\dd t=\frac1{2}\log(b_N/N)+\frac{1}{2}\log(\frac{4\pi}{G''(\gamma)})+O(N^{-\epsilon}).
\end{align*}
On the other hand,
$$ G''(\gamma)=\frac{b_N}{N}\sum_{i=1}^N\frac{1}{(\gamma-\lambda_i)^2}$$
and
$$\sum_{i=1}^N \frac{1}{(\gamma-\lambda_i)^2}=O(N^{3\epsilon})\frac{1}{N^{-2\epsilon}b_N^2}+(N-O(N^{3\epsilon}))\frac{1+O(N^{-\epsilon})}{\gamma^2}=\frac{N}{(b_N/2\beta)^2}(1+O(N^{-\epsilon})).$$
Hence,
$$\log G''(\gamma)=\log \frac{(2\beta)^2}{b_N}+O(N^{-\epsilon}).$$
Putting these results together, we have
\begin{align}
    &\log Z_N=\log C_N+ \frac{N}{2b_N}G(\gamma)+\log\int_{-\infty}^{+\infty} e^{\frac{N}{2b_N}(G(\gamma+ti)-G(\gamma))}\dd t\nonumber
    \\&=\log C_N+ \frac{N\beta}{b_N}\gamma-\frac{1}{2}\sum_{i=1}^N \log(\gamma-\lambda_i)+\frac{1}{2}\log(b_N/N)+\frac{1}{2}\log(\frac{4\pi}{G''(\gamma)})+O(N^{-\epsilon})\nonumber
    \\&=\log C_N+\frac{N}{2}-\frac{N}{2}\log(\frac{b_N}{2\beta})+\frac{1}{2}\log(b_N/N)+\frac{1}{2}\log(\frac{4\pi}{(2\beta)^2/b_N})\label{high_first}
    \\&~~~+\frac{N\beta}{b_N}(\gamma-\frac{b_N}{2\beta})-\frac{N}{2}\log \frac{\gamma}{b_N/2\beta}+O(N^{-\epsilon})-\frac{1}{2}\sum_{i=1}^N\log(1-\frac{\lambda_i}{\gamma}).\label{high_second}
\end{align}
For the term \eqref{high_first},
\begin{align*}
    &\eqref{high_first}=\log \frac{b_N^{N/2-1}\Gamma(N/2)}{2\pi (N\beta)^{N/2-1}}+\frac{N}{2}-\frac{N}{2}\log(\frac{b_N}{2\beta})+\frac{1}{2}\log(\frac{4\pi b_N^2}{(2\beta)^2N})
    \\&~~=\log\frac{\Gamma(N/2)}{2\pi N^{N/2-1}}+\frac{N}{2}+\frac{N}{2}\log 2+\frac{1}{2} \log(\frac{\pi}{N})
    \\&~~=\log\frac{\sqrt{\frac{4\pi}{N}}(\frac{N}{2e})^{N/2}}{2\pi N^{N/2-1}}+\frac{N}{2}\log(2 e)+\frac{1}{2}\log(\pi/N)+O(N^{-1})=O(N^{-1}),
\end{align*}
where we used Stirling's formula
\begin{align*}
    &\Gamma(N/2)=\sqrt{\frac{4\pi}{N}}\left(\frac{N}{2e}\right)^{N/2}(1+O(1/N)),&& \log \Gamma(N/2)=\log \sqrt{\frac{4\pi}{N}}\left(\frac{N}{2e}\right)^{N/2}+O(1/N).
\end{align*}
We next consider the term \eqref{high_second}. The first two terms of \eqref{high_second} can be simplified to
    \begin{align*}
    &\frac{N\beta}{b_N}(\gamma-\frac{b_N}{2\beta})-\frac{N}{2}\log(1+\frac{\gamma-b_N/2\beta}{b_N/2\beta})=\frac{X_N}{2}-\frac{N}{2}\log(1+\frac{X_N}{N}),
\end{align*}
and it converges in probability to $0$, which can be checked by an elementary fact
$$\lim_{N\to \infty} \left( \frac{x}{2}-\frac{N}{2}\log (1+\frac{x}{N})\right)=0 $$ that holds for all $x>0$.
In Lemma \ref{lem:T}, we will prove that $-\frac{1}{2}\sum_{i=1}^N \log(1-\lambda_i/\gamma)$ converges in distribution to some random variable $T$. This shows that $F_N:=\log Z_N$ converges in distribution to the random variable $T$. Finally, $\lim_{N \to \infty} \P(F_1)$ can be easily checked from Proposition \ref{Largest eigenvalue}. This concludes the proof of Proposition \ref{high_temp}.
\end{proof}

In Appendix \ref{app:stat_T}, we prove several the statistical properties of $T$.

We conclude this section by proving that $T_N :=-\sum_{i=1}^N\log(1-\lambda_i/\gamma)$ converges in distribution.
\begin{lem}\label{lem:T}
    The log statistics $T_N=-\sum_{i=1}^N\log(1-\lambda_i/\gamma)$ converges in distribution to some ($N$-independent) random variable.
\end{lem}
\begin{proof}
From Lemma \ref{ls_eig}, $$|\lambda_1|=\max_{i,j}\{|M_{ij}|\}(1+O(b_N^{-1/8})),\ \ \ |\lambda_n|=\max_{i,j}\{|M_{ij}|\}(1+O(b_N^{-1/8}))$$ with high probability, and hence $|\lambda_n| < \gamma$ and $\lambda_1<\gamma$ with high probability. Thus, with high probability, the following expansion holds:
    \begin{align*}
        -\log(1-\lambda_i/\gamma)=\sum_{k=0}^{\infty}\frac{1}{k}(\frac{\lambda_i}{\gamma})^k.
    \end{align*}
	
    Let
    \begin{align*}
        S_k= \sum_{i=1}^N(\frac{\lambda_i}{\gamma})^k,
    \end{align*}
    then Lemma \ref{lem:lim_of_Tr} implies
    \begin{align*}
        S_1=\sum_{i=1}^N \lambda_i/\gamma =Tr(M_N)/\gamma \to 0.
    \end{align*}
    Moreover, from Lemma \ref{lem:trace_square}
    \begin{align*}
        S_2=\sum_{i=1}^N (\frac{\lambda_i}{\gamma})^2= \frac{\sum \lambda_i^2}{b_N^2}\frac{b_N^2}{\gamma^2} \Rightarrow X
    \end{align*}
    for some non-degenerate random variable $X$. Note that, for $k\geq 3$,
    \begin{align*}
        |\sum_{i=1}^N(\frac{\lambda_i}{\gamma})^k|\leq |\frac{\Gamma}{\gamma}|^{k-2}S_2/\gamma^2 
    \end{align*}
    where we let $\Gamma=\max\{|\lambda_n|,\lambda_1\}$. Thus,
    \begin{align*}
        |T_N|\leq \sum_{k=3}^{\infty} \frac{1}{k}|\frac{\Gamma}{\gamma}|^{k-2}S_2\leq -(\frac{\gamma}{\Gamma})^2{S_2}\log(1-\Gamma/\gamma) =: L_N.
    \end{align*}
    The random variable $L_N$ defined above converges to a non-degenerate random variable, since $S_2$ and ${\Gamma}/{\gamma}$ converge in distribution to some random variables.  

We now rewrite $T_N$ as
\begin{align*}
T_N = \sum_{i=1}^N -\log \left(1 - \frac{\lambda_i}{\gamma}\right) - \frac{\lambda_i}{\gamma} - \frac{\lambda_i^2}{2\gamma^2} + S_1+ \frac{1}{2} S_{2}
\end{align*}
and define the random variables
\begin{align*}
U_N &= \sum_{i=1}^N -\log \left(1 - \frac{\lambda_i}{\gamma}\right) - \frac{\lambda_i}{\gamma} - \frac{\lambda_i^2}{2\gamma^2}, \\
U_N^{\epsilon} &= \sum_{| \lambda_i | > \epsilon \gamma} -\log \left(1 - \frac{\lambda_i}{\gamma}\right) - \frac{\lambda_i}{\gamma} - \frac{\lambda_i^2}{2\gamma^2} ,
\end{align*}
\[
U_{N,k} = \sum_{i=1}^k \left(-\log \left(1 - \frac{\lambda_i}{\gamma}\right) - \log \left(1 - \frac{\lambda_{N+1-i}}{\gamma}\right) - \frac{\lambda_i + \lambda_{N+1-i}}{\gamma} - \frac{\lambda_i^2 + \lambda_{N+1-i}^2}{2\gamma^2}\right).
\]
First, we notice that the difference between $U_N$ and $U_N^{\epsilon}$ is bounded by
\begin{align*}
| U_N - U_N^{\epsilon} | &= \left| \sum_{| \lambda_i | < \epsilon \gamma} -\log \left(1 - \frac{\lambda_i}{\gamma}\right) - \frac{\lambda_i}{\gamma} - \frac{\lambda_i^2}{2\gamma^2} \right| \\
&\leq \sum_{| \lambda_i | < \epsilon \gamma} 2 \left| \frac{\lambda_i^2}{\gamma^3} \right| < 2\epsilon \sum \frac{\lambda_i^2}{\gamma^2} = 2\epsilon S_2. 
\end{align*}
By definition, $U_{N, i} \leq U_{N, j}$ for $i \leq j$. Further, for any fixed $k$, if we let $M_k$ be the $k$-th large entry of $\{ |M_{ij}| : 1\leq i \leq j \leq N \}$, then with high probability, $\lambda_k = M_k (1+o(1))$ and $\lambda_{N+1-k} = -M_k (1+o(1))$. (See, e.g., \cite{soshnikov2004poisson}.) We then find that $U_{N, k}$ converges in distribution to some random variable $V_k$.
From the monotonicity of $U_{N, k}$, we find that $V_1 \leq V_2 \leq V_3 \leq \ldots$. Thus, since $U_{N,k} \leq L_N$, from the convergence of $L_N$ we have
\[
\lim_{N \to \infty} V_N = V_{\infty} 
\]
for some random variable $V_{\infty}$.

Let $F_{\epsilon} = \left\{\epsilon < S_2 < \frac{1}{\sqrt{\epsilon}} \right\}$, then
\[
U_{N, 1 / \epsilon} \leq\left( U_N^{\epsilon} | F_{\epsilon}\right) \leq U_{N,\epsilon^{-5 / 2}} \text{ and }|U_N | F_{\epsilon} - U_N^{\epsilon} | F_{\epsilon}|\leq 2 \epsilon S_2 = 2 \sqrt{\epsilon},
\]
which also implies
\[
U_{N, 1 / \epsilon} - 2 \sqrt{\epsilon} \leq U_N | F_{\epsilon} \leq U_{N,\epsilon^{-5 / 2}} + 2 \sqrt{\epsilon}.
\]
Thus, we can obtain 
\[
\P(U_{N, 1 / \epsilon} \geq U + 2 \sqrt{\epsilon}) \leq \P(U_N | F_{\epsilon} \geq U) \leq \P(U_{N,\epsilon^{-5 / 2}}\geq U - 2 \sqrt{\epsilon})
\]
and this implies
\[
\lim_{\epsilon\to 0} \lim_{N \to \infty} \P(U_{N, 1 / \epsilon} \geq U + 2 \sqrt{\epsilon}) = \lim_{\epsilon \to 0} \P(V_{1 / \epsilon} \geq U + 2 \sqrt{\epsilon}) = \P(V_{\infty} \geq U)
\]
\[
\lim_{\epsilon \to 0} \lim_{N \to \infty} \P(U_{N,\epsilon^{-5/2}} \geq U-2\sqrt{\epsilon}) = \lim_{\epsilon \to 0} \P(V_{ \epsilon^{-5/2}} \geq U - 2 \sqrt{\epsilon}) = \P(V_{\infty} \geq U).
\]
Hence, we have
\[
 \P(V_{\infty}\geq U)=\lim_{\epsilon\to 0} \P(V|F_{\epsilon}\geq U) = \lim_{\epsilon \to 0} \lim_{N \to \infty} P(U_N | F_{\epsilon} \geq U) = \P(V_{\infty} \geq U)
\]
We thus conclude that
\[
U_N \Rightarrow V_{\infty}.
\]
This proves that $T_N$ converges to a non-degenerate random variable.
\end{proof}

\section{Low temperature regime} \label{sec:low}
In this section, we assume that the event $F_2=\left\{\lambda_1 > \frac{b_N}{2\beta}\right\}$ holds. Our goal in this section is to prove the following proposition:
\begin{prop}\label{low_temp}
Suppose that assumptions in Theorem \ref{main} hold. Then, given the event $F_2=\left\{\lambda_1 > \frac{b_N}{2\beta}\right\}$, 
    \[
        (\frac{1}{N}\log Z_N)|F_2\Rightarrow \beta X-\frac{1}{2}\log(2e\beta X)
    \] 
    where the random variable $X$ satisfies $X>\frac{1}{2\beta}$ and its distribution is given by
    \[
	\P(X>u)=\frac{1-\exp(-u^{-\alpha})}{1-\exp(-(2\beta)^{\alpha})}.
\]
    Moreover, $\lim_{N \to \infty} \P(F_2) = 1- \exp(-(2\beta)^{\alpha})$.
\end{prop}

We begin by proving the asymptotic behavior of the Taylor coefficients of the function $G$ defined in \eqref{eq:G}.
\begin{lem}\label{low}
Recall the definition of $\gamma$ in Lemma \ref{lem:gamma}. With high probability, there exist ($N$-independent) coefficients $C_1$ and $C_2$ such that the Taylor coefficients $G^{(k)}(\gamma)$ of $G$ at $\gamma$ satisfies
\[
		C_1 N^{k-1}b_N^{-(k-1)}N^{-k\epsilon} \leq G^{(k)}(\gamma) \leq C_2 N^{k-1}b_N^{-(k-1)}.
\]
\end{lem}

\begin{proof}
Recall we proved in Lemma \ref{lem:gamma} that
\[
	\gamma-\lambda_1= \frac{1}{2\beta-b_N/\lambda_1}\frac{b_N}{N}+O(b_N N^{-1-\epsilon}).
\]
We invoke Lemma \ref{lem:gap_eig} as in the proof of Lemma \ref{lem:high_tay} to find that $\lambda_1>\frac{b_N}{2\beta }(1+N^{-\epsilon})$ with high probability. Thus, we have
\[
    \gamma-\lambda_1= O(b_N N^{-1+\epsilon}).
\] 
Further, since $2\beta-b_N/\lambda_1\leq 2\beta$, we have
\[
    \gamma-\lambda_1 \geq \frac{b_N}{2\beta N}+O(\frac{b_N}{N}N^{-\epsilon}) \geq \frac{b_N}{4\beta N}.
\]

From the definition of $G^{(k)}(\gamma)$, we find an upper bound for $G^{(k)}(\gamma)$,
    \begin{align*}
        G^{(k)}(\gamma)= \frac{b_N}{N}\sum_i \frac{(k-1)!}{(\gamma-\lambda_i)^k}\geq \frac{b_N}{N}\frac{1}{(\gamma-\lambda_1)^k} = \Theta(N^{k-1}b_N^{-k+1}N^{-k\epsilon}).
    \end{align*}
For the lower bound for $G^{(k)}(\gamma)$, since $\lambda_1-\lambda_2> b_NN^{-\epsilon}$ with high probability due to Lemma \ref{lem:gap_eig},
    \begin{align*}
        \begin{split}
             G^{(k)}(\gamma)&= \frac{b_N}{N}\sum_i (k-1)!\frac{1}{(\gamma-\lambda_i)^k}\\
             &= O\left(\frac{b_N}{N}\frac{1}{(\gamma-\lambda_1)^k}\right)+O\left(N\frac{b_N}{N}b_N^{-k}N^{k\epsilon}\right)=O(N^{k-1}b_N^{-k+1}).
        \end{split}
    \end{align*}
This completes the proof of the desired lemma.
\end{proof}

In order to prove Proposition \ref{low_temp}, as in the high-temperature regime, we consider the integration representation formula,
\begin{align*}
    \log Z_N=\log C_N+ \frac{N}{2b_N}G(\gamma)+\log\int_{-\infty}^{+\infty} e^{\frac{N}{2b_N}(G(\gamma+ti)-G(\gamma))}\dd t.
\end{align*}
The key part of the proof is to show that
\[
	\lim_{N\to \infty}\frac{1}{N}\log\int_{-\infty}^{+\infty} e^{\frac{N}{2b_N}(G(\gamma+ti)-G(\gamma))}\dd t=0.
\]
To show this, we will prove the following claim:
\[
	C_1 N^{c_1}\leq\int_{-\infty}^{+\infty} e^{\frac{N}{2b_N}(G(\gamma+ti)-G(\gamma))}\dd t\leq C_2 N^{c_2}
\]
for some $C_1,C_2>0$ and $c_1,c_2\in \R$. For the upper bound in the claim, we notice
  \begin{align*}
    ~~&\left|\int_{-\infty}^{+\infty} e^{\frac{N}{2b_N}(G(\gamma+ti)-G(\gamma))}\dd t\right|=\left |\int_{0}^{+\infty} 2\re e^{\frac{N}{2b_N}(G(\gamma+ti)-G(\gamma))}\dd t\right|
    \\&\leq \int_{0}^{+\infty} 2 \left|e^{\frac{N}{2b_N}(2\beta ti)-\frac{1}{2}\sum\log(1+\frac{ti}{\gamma-\lambda_i})}\right|\dd t
    = \int_{0}^{+\infty} 2e^{-\frac{1}{4}\sum\log(1+\frac{t^2}{(\gamma-\lambda_i)^2})}\dd t
    \\&\leq \int_{0}^{b_N^4}1\dd t+\int_{b_N^4}^{\infty}\left(\frac{t^2}{\lambda_1^2}\right)^{-N/8}\dd t=O(b_N^4),
\end{align*}
where we used the Taylor expansion of $G(\gamma+t)$ with Lemma \ref{low} as
\begin{align*}\label{taylor}
    G(\gamma+t)=\sum \frac{1}{k!}G^{(k)}(\gamma)t^k= \sum O(b_N^{-(k-1)}N^{(k-1)})t^k=G(\gamma)+\frac{b_N}{N}\sum_{k\geq 2} O((\frac{N}{b_N}t)^k),
 \end{align*}
which converges for $|t|<\frac{b_N}{N}$.

To prove the lower bound of the claim, we will deform the integral by considering the following curve: let $\Gamma$ be the curve that starts from $\gamma$ and ends at $-\infty$ such that for any point $t \in \Gamma$, we have $\im(G(t))=0$. (Implicit function theorem ensures the existence of such a curve.) By the Cauchy-Riemann equation, $\re(G(z))$ is monotone along the curve. Since $G(z)$ decreases near $\gamma$ as imaginary part increases from its Taylor expansion, $G(z)=\re(G(z))$ decays along the curve. We further decompose $\Gamma$ by $\Gamma = \Gamma^+ \cup \Gamma^-$, where $\Gamma^{+} := \Gamma\cap \C^+$ and $\Gamma^{-} :=\Gamma\cap \C^-$. Note that both $\Gamma^+$ and $\Gamma^-$ start from $\gamma$ and end at $-\infty$. Let $C_R$ be an arc starting from $\{|z|=R\}\cap \Gamma^+$ to $\{|z|=R\}\cap \{\re(z)= \gamma\}$. Then for any sufficiently large $R$ with $R>2\lambda_1$, we have
$$\re(G(z))=\re(2\beta z -\frac{b_N}{N}\sum_{i}\log(z-\lambda_i))\leq 2\beta \gamma -b_N\log(R/2).$$
Then, the contour integral along the curve $C_R$ satisfies 
$$
    |\int_{C_R} e^{\frac{N}{2b_N}G(z)}\dd z|\leq \frac{e^{N\beta\gamma/b_N}\pi R}{(R/2)^{N/2}}.
$$
In particular, as $R \to \infty$ the integral above converges to $0$. This enables us to change the integral along $\re(z)=\gamma$ to the integral along the curve $\Gamma$, i.e.,
$$
    \int_{\gamma-i\infty}^{\gamma+i\infty} e^{\frac{N}{2b_N}G(z)}\dd z =\int_{\Gamma}e^{\frac{N}{2b_N}G(z)}\dd z=\int_{\Gamma^+}e^{\frac{N}{2b_N}G(z)}\dd z+\int_{\Gamma^-}e^{\frac{N}{2b_N}G(z)}\dd z=2\int_{\Gamma^+}e^{\frac{N}{2b_N}G(z)}\dd z,
$$
where we also used the fact $\overline{G(\bar{z})}=G(z)$.

Let 
$$
    K=-i \int_{\gamma-i\infty}^{\gamma+i\infty} \exp(\frac{N}{2b_N}(G(z)-G(\gamma)))\dd z=-2i\int_{\Gamma^+}\exp(\frac{N}{2b_N}[G(z)-G(\gamma)])\dd z,
$$
Then, by letting $z = x+ iy$,
$$
    K=-2i\int_{\Gamma^+}\exp(\frac{N}{2b_N}[G(z)-G(\gamma)])\dd x+2\int_{\Gamma^+}\exp(\frac{N}{2b_N}[G(z)-G(\gamma)])\dd y.
$$
To obtain a lower bound for $K$, we will prove the lower bound for the second term in the right side of the equation above. To estimate this term, we need to prove some properties of the contour $\Gamma^+$. We notice that the Taylor expansion
$$
    G(\gamma+z)= G(\gamma)+ \sum_{k\geq 2}\frac{1}{k!}G^{(k)}(\gamma)z^k
$$
holds for $|z|<\frac{b_N}{N}$. Thus, for $z = x+ iy \in \Gamma^+$,
$$
    0=\im G(\gamma+x+iy)= xy G''(\gamma)-\frac{1}{6}G^{'''}(\gamma)y^3+\frac{1}{2}G^{'''}(\gamma)x^2y+\sum_{k\geq 4}\frac{1}{k!}G^{(k)}(\gamma)\im(x+iy)^k
$$
and
$$
    |\im (x+iy)^k|\leq k|y||x+iy|^{k-1}.
$$
We then find for $|z|\leq \frac{1}{N^2}$ that
\begin{align*}
    \Omega(x,y)&=\frac{1}{G^{''}(\gamma) y}\sum_{k\geq 4}\frac{1}{k!}G^{(k)}(\gamma)\im(x+iy)^k\leq \sum_{k\geq 4}\frac{1}{(k-1)!}\frac{G^{(k)}(\gamma)}{G^{''}(\gamma)}|x+iy|^{k-1}\\&\leq O\left(\frac{N^{2+2\epsilon}}{b_N^{2}}\right)(x^2+y^2)|z|\leq O(\frac{N^{2\epsilon}}{b_N^2})(x^2+y^2).
\end{align*}
Set $a=-\frac{G'''(\gamma)}{G''(\gamma)}>0$. From Lemma \ref{low}, we find $a=O(\frac{N}{b_N}N^{2\epsilon})$ and then
\begin{align*}
    0&=x+\frac{1}{6}a y^2-\frac{1}{2}a x^2+O(\frac{N^{2\epsilon}}{b_N^2})(x^2+y^2)
    \\&= x+\frac{1}{6}a y^2(1+o(1))-\frac{1}{2}a x^2(1+o(1))=x(1+o(1))+\frac{1}{6}a y^2 (1+o(1)).
\end{align*}
Thus, 
\begin{align*}
    x=-\frac{1}{6}a y^2(1+o(1))\text{ for } |x+iy|\leq \frac{1}{N^2}
\end{align*}
and $y$ increases along contour $\Gamma^+$ starting from $0$ in $|z|<\frac{1}{N^2}$. Moreover, since 
\[
	|z|^2= x^2+y^2= y^2 + O(\frac{N^{2+4\epsilon}}{b_N^2} y^4),
\]
with high probability, we have
$$
    y=\Theta(|z|),x= \Theta\left(\frac{G'''(\gamma)}{G''(\gamma)}|z|^2\right) \text{ for } |z|<\frac{1}{N^2}.
$$

We now go back to the lower bound for
$$
    2\int_{\Gamma^+}\exp(\frac{N}{2b_N}[G(z)-G(\gamma)])\dd y.
$$
We decompose $\Gamma^+$ to $\Gamma_1 \cup \Gamma_2 \cup \Gamma_3$, where we let
$$\Gamma_1=\{|z|\leq N^{-3}\}\cap \Gamma^+ , \Gamma _2=\{N^{-3}\leq |z|\leq N^{-2}\}\cap \Gamma^+,\Gamma_3=\{|z|\geq N^{-2}\}\cap \Gamma^+$$
$$z_3= \{|z|=N^{-3}\}\cap \Gamma^+ ,z_2=\{|z|=N^{-2}\}\cap \Gamma^+.$$
In Lemma \ref{great0}, we will show that for any real valued function $f(z)$ on $\Gamma^+$ that decreases along the curve as $z$ moves from $\gamma$ to $-\infty$, $\int_{\Gamma^+} e^{f(z)} \dd y\geq 0$. With this lemma, if we let
\begin{align*}
    g(z)=\begin{cases}
        \frac{N}{2b_N}[G(z)-G(\gamma)], & |z|>N^{-2},
        \\ \frac{N}{2b_N}[G(z_2)-G(\gamma)] ,& |z|<N^{-2},
    \end{cases}
\end{align*}
Then $g(z)$ decreases along the curve $\Gamma^+$ and
\begin{align*}
    &\int_{\Gamma^+}\exp(\frac{N}{2b_N}[G(z)-G(\gamma)])\dd y\geq\int_{\Gamma_1}\exp(\frac{N}{2b_N}[G(z_3)-G(\gamma)])\dd y+\int_{\Gamma_2\cup \Gamma_3}\exp(g(z)))\dd y
    \\&\geq \int_{\Gamma_1}\exp(\frac{N}{2b_N}[G(z_3)-G(\gamma)])-\exp(\frac{N}{2b_N}[G(z_2)-G(\gamma)])\dd y+\int_{\Gamma^+}\exp({g(z)})\dd y
    \\&\geq \int_{\Gamma_1}\exp(\frac{N}{2b_N}[G(z_3)-G(\gamma)])-\exp(\frac{N}{2b_N}[G(z_2)-G(\gamma)])\dd y.
\end{align*}
We need to find a lower bound for the right-side of the equation above. Considering again the Taylor expansion of $G(\gamma+z)$ with the estimates in Lemma \ref{low}, we get
\[
	G(z_2)-G(\gamma) \leq -C_1 \frac{N^{1-2\epsilon}}{b_N}N^{-4}, \quad G(z_3)-G(\gamma) \geq -C_2\frac{N}{b_N}N^{-6}
\]
for some constants $C_1, C_2 > 0$. Since
\begin{align*}
    \int_{\Gamma_1} \dd y \geq \im z_3-\im \gamma \geq C_3 N^{-3}
\end{align*}
for some constant $C_3 > 0$, we find that
\begin{align*}
    \int_{\Gamma_1}\exp(\frac{N}{2b_N}[G(z_3)-G(\gamma)])-\exp(\frac{N}{2b_N}[G(z_2)-G(\gamma)])\dd y 
    \geq C_4 N^{-3}\frac{N}{2b_N}\frac{N^{1-2\epsilon}}{b_N}N^{-4}
\end{align*}
for some constant $C_4>0$. This proves the desired lower bound for $K$ and hence the desired claim. 

\begin{lem}\label{great0}
    Let $f$ be a real-valued function defined on $\Gamma^+$ and $f(z)$ is decreasing along the curve $\Gamma^+$ as z moves from $\gamma$ to $-\infty$
    \begin{align*}
        \int_{\Gamma^+}e^{f(z)}\dd y\geq 0.
    \end{align*}
\end{lem}
\begin{proof}[Proof of Lemma \ref{great0}]
Consider the function $s:[0,1]\to \mathbb{C}$ such that  $s(0)= \gamma$, $s(1)=-\infty$ and which follows the contour of  $\Gamma^+$, then we have
\begin{align*}
    \int_{\Gamma^+}e^{f(z)}\dd y=\int_{0}^1e^{f(s(t))}\im s'(t) \dd t.
\end{align*}
As $t$ moves from 0 to 1, we can divide the interval $[0,1]$ into $0=s_0<s_1<...,<s_k=1$ such that $\im s(t)$ increases on $[s_{2m},s_{2m+1}]$ and decreases on $[s_{2m+1},s_{2m+2}]$. \
Let $\Delta_k= s(s_{k+1})-s(s_k)$,then we have
\begin{align*}
    \int_{s_{2m}}^{s_{2m+1}}e^{f(s(t))}\im s'(t) \dd t\geq  e^{s_{2m+1}}(s(s_{2m+1})-s(s_{2m})),
\end{align*}
\begin{align*}
    \int_{s_{2m-1}}^{s_{2m}}e^{f(s(t))}\im s'(t) \dd t\geq -e^{s_{2m-1}}(s(s_{2m})-s(s_{2m-1})).
\end{align*}
Therefore,
\begin{align*}
    \int_{\Gamma^+}e^{f(z)}\dd y&=\int_{0}^1e^{f(s(t))}\im s'(t) \dd t\geq \sum e^{s_{2m+1}}(\Delta_{2m}-\Delta_{2m+1})
    \\&\geq e^{f(s_1)}(\Delta_0-\Delta_1)+\sum_{m\geq 1} e^{s_{2m+1}}(\Delta_{2m}-\Delta_{2m+1})
    \\&\geq e^{f(s_1)}(\Delta_0-\Delta_1)+e^{f(s_3)}(\Delta_2-\Delta_3)+\sum_{m\geq 2} e^{s_{2m+1}}(\Delta_{2m}-\Delta_{2m+1})
    \\&\geq e^{f(s_3)}((\Delta_0-\Delta_1)+(\Delta_2-\Delta_3))+\sum_{m\geq 2} e^{s_{2m+1}}(\Delta_{2m}-\Delta_{2m+1})
    \\&\geq e^{f(s_{2k+1})}(\sum_{m=0}^k \Delta_{2m}-\Delta_{2m+1})+\sum_{m\geq k+1} e^{s_{2m+1}}(\Delta_{2m}-\Delta_{2m+1})
    \\&\geq e^{f(s_{2t+1})}\sum \Delta_{2m}-\Delta_{2m+1}\geq 0.
\end{align*}
This comes from the fact that $\Gamma^+$ does not across the real line and it implies $\sum_{m=0}^k \Delta_{2m}-\Delta_{2m+1}\geq 0$ for arbitrary $k\geq 0$.

\end{proof}
Now, let us get back to the proof of Proposition \ref{low_temp}. 
\begin{proof}[Proof of Proposition \ref{low_temp}]
The logarithm of partition function
\begin{align*}
    \log Z_N=\log C_N+\frac{N\beta \gamma}{b_N}-\frac{1}{2}N\log b_N-\frac{1}{2}\sum_{i=1}^N \log (\gamma/b_N-\lambda_i/b_N)+\log \int_{-\infty}^{\infty}e^{\frac{N}{2b_N}(G(\gamma+it)-G(\gamma)}\dd t.
\end{align*}
Since we proved $$\lim_{N\to\infty }\frac{1}{N}\log \int_{-\infty}^{\infty}e^{\frac{N}{2b_N}(G(\gamma+ti)-G(\gamma))}=0,$$
the limiting distribution of the free energy $F_N=\frac{1}{N}\log Z_N$ is equal to that of
\begin{align*}
     &\frac{1}{N}\log C_N-\frac{1}{2}\log b_N+\frac{\beta\gamma}{b_N}-\frac{1}{2N}\sum_{i=1}^N\log(\gamma/b_N-\lambda_i/b_N)
    \\&= \frac{1}{N}\left[\log \frac{\Gamma(N/2)b_N^{N/2-1}}{2\pi (N\beta)^{N/2-1}b_N^{N/2}}-\frac{1}{2}\sum_{i=1}^N\log(\gamma/b_N-\lambda_i/b_N)\right]+\frac{\beta \gamma}{b_N}.
\end{align*}
Stirling's approximation implies
    \[
        \Gamma(N/2)=\sqrt{\frac{4\pi}{N}}\left(\frac{N}{2e}\right)^{N/2}(1+O(1/N)),
    \]
    hence
$$\lim \frac{1}{N}\log \frac{\Gamma(N/2)}{2\pi (N\beta)^{N/2-1}b_N}=-\frac{1}{2}\log(2e\beta).$$
On the other hand,
\begin{align*}
\sum \log(\gamma/b_N-\lambda_i/b_N)&=\sum_{|\lambda_i|<b_NN^{-\epsilon}}\log(\gamma/b_N-\lambda_i/b_N)+\sum_{|\lambda_i|>b_NN^{-\epsilon}}\log(\gamma/b_N-\lambda_i/b_N) 
\\&=(N-O(N^{3\epsilon}))\log(\gamma/b_N)(1+O(N^{-\epsilon}))+O(N^{3\epsilon})O(\log N).
\end{align*}
This implies that the limit of $\frac{1}{N}\sum_{i=1}^N \log (\gamma/b_N-\lambda_i/b_N)$
is governed by the leading order term $\log (\lambda_1/b_N).$
Together with all above results, the limiting distribution of free energy
\begin{align}
     F_N&=\frac{1}{N}\log \frac{\Gamma(N/2)}{2\pi (N\beta)^{N/2-1}b_N}-\frac{1}{2N}\sum_{i=1}^N\log(\gamma/b_N-\lambda_i/b_N)+\frac{\beta\lambda_1}{b_N}
     \\&+\frac{1}{N}\log \int_{-\infty}^{\infty}e^{\frac{N}{2b_N}(G(\gamma+ti)-G(\gamma))}\nonumber
\end{align}
is equal to
\begin{align}
    -1/2\log(2e\beta\lambda_1/b_N)+\frac{\beta\lambda_1}{b_N}\Rightarrow -\frac{1}{2}\log \beta X+\beta X-1/2\log(2e),
\end{align}
where $X$ is the limiting distribution of $\frac{\lambda_1}{b_N}$. Finally, $\lim_{N \to \infty} \P(F_2)$ can be easily checked from Proposition \ref{Largest eigenvalue}. This concludes the proof of Proposition \ref{low_temp}. 
\end{proof}

\appendix

\section{Proofs of technical lemmas} \label{app:proof}

\begin{proof}[Proof of Lemma \ref{inf_norm}]
We consider the cases $1<\alpha<2$ and $0<\alpha<1$ separately.

\textit{Case 1:} $1<\alpha<2$.
    
    We choose $\epsilon=\frac{1}{2M+1}$, where M is an integer satisfying $\epsilon<\frac{1}{4}-\frac{\alpha}{8}$.
    Then,
    $$\E(\#\{1\leq j \leq N: |M_{ij}|\geq b_N^{k\epsilon}\})=NG(b_N^{k\epsilon}).$$
    For binomial random variable $X\sim B(N,p)$, Chernoff's inequality implies $\P(X\geq \E(X)+t)\leq \exp(-\frac{t^2}{2Np}+\frac{t^3}{6(Np)^2})$. This implies 
    $$\P(\#\{1\leq j \leq N : |M_{ij}|\geq b_N^{k\epsilon}\}\geq 2N G(b_N^{k\epsilon}))\leq \exp(-\frac{NG(b_N^{k\epsilon})}{3})\leq \exp(-N^{\gamma})$$
    for $k\leq M$ and $\gamma=\frac{1}{4M+2}$.
    We apply this for every row and we have
    \beq\label{norm_eq}
    \P(\#\{1\leq j \leq N : |M_{ij}|\geq b_N^{k\epsilon}\}\geq 2N G(b_N^{k\epsilon})\text{ for some } i) \leq N \exp(-N^{\gamma}).
    \eeq
    Therefore,
    \begin{align}
        &\sum_{j:|M_ij|\leq b_N^{\frac{M+1}{2M+1}}} |M_{ij}|\leq \sum_{k=0}^M \#(1\leq j \leq N: |M_{ij}|\geq b_N^{k\epsilon})b_N^{(k+1)\epsilon}
        \\ &\leq \sum_{k=0}^M b_N^{(k+1)\epsilon}2N G(b_N^{k\epsilon})=\sum_{k=0}^M b_N^{(k+1)\epsilon}2NL(b_N^{k\epsilon})b_N^{-\alpha k \epsilon}\leq C_1 b_N^{\epsilon}.
    \end{align}
    This implies the first case.
		
\textit{Case 2:} $0<\alpha\leq 1$.

    We again choose $\epsilon=\frac{1}{2M+1}$, where $\epsilon<\frac{\alpha}{8}$ and $\gamma=\frac{1}{4M+2}$.
    Since \eqref{norm_eq} also holds in this case,
    \begin{align}
        &\sum_{j:|M_{ij}|\leq b_N^{\frac{M+1}{2M+1}}} |M_{ij}|\leq \sum_{k=0}^M \#(1\leq j \leq N: |M_{ij}|\geq b_N^{k\epsilon})b_N^{(k+1)\epsilon}
        \\ &\leq \sum_{k=0}^M b_N^{(k+1)\epsilon}2N G(b_N^{k\epsilon})=\sum_{k=0}^M b_N^{(k+1)\epsilon}2NL(b_N^{k\epsilon})b_N^{-\alpha k \epsilon}\leq C_2 b_N^{\epsilon} 2N b_N^{\frac{M(1-\alpha)}{2M+1}}\leq C_2 b_N^{\frac{1}{2}+2\epsilon}.
    \end{align}
\end{proof}
\begin{proof}[Proof of Lemma \ref{ls_eig}]
 Given any $\epsilon > 0$, we have the following:
 \begin{align*}
&\P(\max_{1\leq i\leq j \leq N}|M_{ij}|< b_N^{1-\epsilon})=(1-\P(|M_{11}|>b_N^{1-\epsilon}))^{N(N+1)/2}
 \\&=O(\exp(-{\frac{N(N+1)}{2} \P(|M_{11}|>b_N^{1-\epsilon}) })=O(\exp(-\frac{N(N+1)}{2} b_N^{(1-\epsilon)(-\alpha+\epsilon')})))
 \\&=O(\exp(-N^{\delta}))
 \end{align*} for some $\delta>0$.
 This implies that $$\max_{1\leq i\leq j \leq N}|M_{ij}|>b_N^{1-\epsilon}$$ holds with high probability for every $\epsilon>0$.
 Coupled with Lemma \ref{inf_norm}, this result leads to the following bound on the extreme eigenvalues of $M$:
  \[
		|\lambda_N|, \lambda_1 \leq \| M \|_{\infty}\leq \max_{1\leq i\leq N}\sum_{1\leq j\leq N} |M_{ij}|=\max_{1\leq i\leq j \leq N}|M_{ij}|(1+O(b_N^{-1/8})).
	\]
	This proves the upper bounds of the desired lemma. For the lower bounds, we let $e_k$ be the $k$-th standard basis vector and let $f_1=\frac{1}{\sqrt{2}} (e_i+e_j), f_2=\frac{1}{\sqrt{2}}(e_i-e_j)$. Then $|f_1^* M f_1|, |f_2^* M f_2|\geq \max |M_{ij}|(1+O(b_N^{-1/8}))$. Moreover, it is not hard to see that the signs of $f_1^* M f_1, f_2^* M f_2$ are different.
 If $f_1^*Mf_1>0$, then $\lambda_1\geq f_1^*Mf_1,\lambda_n\leq f_2^*Mf_2$. If $f_1^*Mf_1<0$,then $\lambda_1\geq f_2^*Mf_2,\lambda_n\leq f_1^*Mf_1$.
 This and with high probability conditions about each rows in Lemma \ref{lem:whp_set1} conclude the proof of the lemma.
\end{proof}
\begin{proof}[Proof of Lemma \ref{lem:gap_eig}]
    Let $$M_{0}=\max_{1\leq i,j \leq N}{|M_{ij}|}$$
\begin{align*}
    \P(\lambda_1<b_N N^{-\epsilon})&\leq \P(\lambda_1< b_N N^{-\epsilon}, \lambda_1=M_0(1+O(b_N^{-1/8}))  )+\P(\lambda_1 \neq M_0(1+O(b_N^{-1/8})))
    \\&\leq \P(M_0<2b_N N^{-\epsilon})+\P(\lambda_1 \neq M_0(1+O(b_N^{-1/8})))
\end{align*}
Since Lemma \ref{lem:a_N} implies $b_N\geq N^{2/\alpha -\epsilon}$ for every $\epsilon>0$,
\begin{align*}
    \P(M_0<2b_N N^{-\epsilon})=(1-G(b_N N^{-\epsilon}))^{N(N+1)/2}=O(e^{-N^{\epsilon\alpha/2}})
\end{align*}
and due to Lemma \ref{ls_eig}
$$\P(\lambda_1 \neq M_0(1+O(b_N^{-1/8})))< N^{-\delta}$$ for some $\delta>0$.
These imply $$\P(\lambda_1<b_N N^{-\epsilon})<2N^{-\delta}$$ and this implies 
$\lambda_1>b_N N^{-\epsilon}$ holds with high probability for every $\epsilon>0$.
\\For $x<y<x(1+Cx^{-\epsilon})$,
\begin{align*}
   &(1-G(y))^{{N(N+1)}/{2}}-(1-G(x))^{N(N+1)/2}\leq (G(x)-G(y))\frac{N(N+1)}{2}(1-G(x))^{N(N+1)/2-1}
   \\& \ \ \ \leq\frac{G(x)-G(y)}{G(x)}\frac{N(N+1)}{2}G(x)(1-G(x))^{N(N+1)/2-1}
\end{align*}
Since $e^{x^{\delta}}L(x)$ is increasing function for $\delta=\epsilon/2$, 
$L(y)/L(x)\geq e^{x^{\delta}-y^{\delta}}\geq e^{\delta(x-y)x^{\delta-1}}$.
\\Hence,
$$\frac{G(y)}{G(x)}\geq \exp(\alpha \log x-\alpha \log y+\delta(x-y)x^{\delta-1} )\geq \exp(-C'x^{-\epsilon/2}).$$
This implies
$$\frac{G(x)-G(y)}{G(x)}\leq 1- \exp(-C'x^{-\epsilon/2})=O(x^{-\epsilon/2}).$$
Also, 
$$\lim_{N\to \infty}\frac{N(N+1)}{2}G(x)(1-G(x))^{N(N+1)/2-1}=ue^{-u} $$
where
$$ \lim_{N\to\infty} \frac{N(N+1)}{2}G(x)=u.$$
Since $ue^{-u}$ is bounded, $\frac{N(N+1)}{2}G(x)(1-G(x))^{N(N+1)/2-1}$ is also bounded. 
\\This means 
\beq \P(x<M_0<y)=O(x^{-\epsilon/2}) \label{qua_eig}\text{ for }y=x(1+O(x^{-\epsilon})).\eeq
\\For every $\epsilon>0$, since $\lambda_1> b_N N^{-\epsilon/2}$ holds with high probability, $\lambda_1-\lambda_2< b_N N^{-\epsilon}$ implies $\lambda_2> b_N N^{-\epsilon/2}-b_N N^{-\epsilon}$, this and \eqref{qua_eig} implies
\begin{align*}
    &\P(\lambda_1-\lambda_2<b_NN^{-\epsilon})=\P(\lambda_2< \lambda_1< \lambda_2+ b_N N^{-\epsilon})
    \\& \ \ \leq \P(\lambda_2< M_0(1+O(b_N^{-1/8}))< \lambda_2+ b_N N^{-\epsilon})+\P(\lambda_1\neq M_0(1+O(b_N^{-1/8})))=O(N^{-\epsilon'})
\end{align*}
for some $\epsilon'>0$.
\\Similarly, \eqref{qua_eig} implies
\begin{align*}
    &\P(Cb_N(1-N^{-\epsilon})<\lambda_1< Cb_N(1+N^{-\epsilon}))
    \\&\leq \P(Cb_N(1-N^{-\epsilon})<M_0(1+O(b_N^{-1/8}))< Cb_N(1+ N^{-\epsilon}))+ \P(\lambda_1\neq M_0(1+O(b_N^{-1/8})))
    \\&=O(N^{-\epsilon'})
\end{align*}
for some $\epsilon'>0$.
The special case $C=\frac{1}{2\beta}$ is our desired result. 

\end{proof}
\begin{proof}[Proof of Lemma \ref{lem:a_N}]
In order to prove the first part of the lemma, it suffices to show that for any $0<\epsilon<\alpha$ there exists $C_\epsilon>0$ such that $C_{\epsilon} t^{-\alpha-\epsilon}<\P(|X|>t)<C_{\epsilon} t^{-\alpha+\epsilon}$ holds for any $t>C_{\epsilon}$. For a given $\epsilon>0$, we can easily see from the definition of heavy-tailed random variable in Definition \ref{stable law} that there exists $t_0>0$ such that
\[
	\P(|X|>2t)/\P(|X|>t)<2^{-\alpha+\epsilon}
\]
holds for any $t>t_0$. For any $y>2t_0$ we can choose an integer $m$ such that $t_0<y/2^m\leq 2t_0$. Applying the inequality above $m$-times, we find that
\[
	\P(|X|>y)/\P(|X|>y/2^m)\leq 2^{m(-\alpha+\epsilon)}\leq (2t_0/y)^{\alpha-\epsilon}=y^{-\alpha+\epsilon}(2t_0)^{\alpha-\epsilon}.
\]
Hence, we obtain
\[
	\P(|X|>y)<\P(|X|>y/2^m)y^{-\alpha+\epsilon}(2t_0)^{\alpha-\epsilon}=\P(|X|>t_0)(2t_0)^{\alpha-\epsilon}y^{-\alpha+\epsilon}=C_{\epsilon}y^{-\alpha+\epsilon},
\]
which proves the desired upper bound. The proof for the lower bound is similar and we omit it here.

The second part of the lemma obviously holds due to the first part. This concludes the proof of the lemma.
\end{proof}
\begin{proof}[Proof of Lemma \ref{lem:X^delta}]
    Let $Y=X\mathbbm{1}_{\{X\leq a_N\}}$, $Z=X\mathbbm{1}_{\{X>a_N\}}$ and let $f$ be the probability density function of X and 
    $F(x)=\int_{x}^{\infty} f(s)\dd s$
    $$\E(Y^{\delta})=\int_{0}^{a_N}x^{\delta}f \dd x=[-x^{\delta}F(x)]_{0}^{a_N}+\int_{0}^{a_N} \delta x^{\delta-1}F(x) \dd x$$
    For $\epsilon< \alpha-\delta$
    $$\int_{0}^{a_n}\delta x^{\delta-1}F(x)\dd x\leq \int_{0}^{C_{\epsilon}}\delta x^{\delta-1}\dd x +\int_{C_{\epsilon}}^{a_n} C_{\epsilon} \delta x^{\delta-1} x^{-\alpha+\epsilon}=O(1)$$
    where $C_{\epsilon}$ is such that for $\epsilon>0$ which satisfies $\P(X>t)\leq C_{\epsilon} t^{-\alpha+\epsilon} $ for every $t>C_{\epsilon}$
    $$\E(Z^{\delta})=\int_{a_N}^{\infty} x^{\delta} f\dd x=[-x^{\delta}F]_{a_N}^{\infty}+\int_{a_N}^{\infty} \delta x^{\delta-1}F(x)\dd x$$
    $\lim_{x\to \infty} x^{\delta}F(x)\leq \lim_{x\to\infty} C_{\epsilon}x^{\delta}x^{-\alpha+\epsilon}=0$ for $\epsilon< \alpha-\delta$
    and 
    $$\int_{a_N}^{\infty} \delta x^{\delta-1}F(x) \dd x\leq \int_{a_N}^{\infty} \delta C_{\epsilon} x^{\delta-1}x^{-\alpha+\epsilon} \dd x=O(a_N^{-\alpha+\delta+\epsilon}).$$
\end{proof}
\begin{proof}[Proof of Lemma \ref{lem:sum_eig_sq_bound}]
    Since
    $$\sum_{i} \lambda_i^2=\sum_{i,j=1}^N M_{ij}^2$$
    and $M_{ij}^2$ follows $\frac{\alpha}{2}$-stable laws and $\alpha/2<1$ holds,
    \\ for $\delta< \alpha/2<1$ 
    $$\E((\sum_{i} \lambda_i^2)^{\delta})=O(N^2)$$
    due to Lemma \ref{lem:sum X^delta}.
    Hence, for every $\epsilon>0$
    $$\P(\sum_{i} \lambda_i^2> N^{4/\alpha+\epsilon})N^{\delta (4/\alpha+\epsilon)}\leq \E((\sum_{i} \lambda_i^2)^{\delta})=O(N^2)$$
    Choose $\delta= \alpha/2 -\epsilon/100$ and we have
    $$\P(\sum_{i} \lambda_i^2> N^{4/\alpha+\epsilon})= O(N^{-\alpha\epsilon/4}).$$
\end{proof}
\begin{proof}[Proof of Lemma \ref{lem:num_eig}]
    Since above lemma and $b_N<N^{2/\alpha+\epsilon/10}$ with high probability due to Lemma \ref{lem:sum_eig_sq_bound}, $$\P(\sum_{i} \lambda_i^2 >b_N^2 N^{\epsilon})= O(N^{-\alpha\epsilon/8})$$
	holds with high probability.
  \\ Hence, with high probability, we obtain
    $$\sum_{i} \lambda_i^2\leq b_N^2 N^{\epsilon}$$ 
    and this implies
    $$\#\{ |\lambda_i|>b_N N^{-\epsilon}\}=O(N^{3\epsilon}).$$  
\end{proof}
\begin{proof}[Proof of Lemma \ref{lem:trace_square}]
    For the sum of the square of eigenvalues, we have
    $$\sum_{i} \lambda_i^2= \Tr(MM^*)=\sum_{i,j} a_{ij}^2,$$
    where we set $a_{ij}=M_{ij}$.
    Let $$c_N= \inf \{t: \P(a_{ij}^2>t)<\frac{2}{N(N+1)}\}=b_N^2.$$ 
    By Proposition \ref{stable_dist}, for $$d_N=\frac{N(N+1)}{2}\E(a^2_{11}\mathbbm{1}_{a_{11}^2\leq c_N})=\frac{N(N+1)}{2}\E(a_{11}^2 \mathbbm{1}_{|a_{11}|\leq b_N }).$$
    Since $a_{ij}^2$ follows $\alpha/2<1$ stable law, the remark implies $$d_N/c_N \to c$$ for some constant c.
    Since $$\sum_{i,j=1}^N a_{ij}^2= 2\sum_{1\leq i\leq j\leq N} a_{ij}^2-\sum_i a_{ii}^2,$$
    this implies
    $$ \frac{\sum_{i\leq j }a_{ij}^2-d_N}{c_N}\Rightarrow Y$$ for some non-degenerate random variable Y.
    
    Hence,
    $$\frac{\sum \lambda_i^2}{b_N^2}\Rightarrow X$$ for some non-degenerate random variable X.
\end{proof}
\begin{proof}[Proof of Lemma \ref{lem:lim_of_Tr}]
Let us assume that each entry follows a $\alpha$-stable random variable X and $a_N=\inf\{u:1-\P(|X|>u)\leq 1/N\}$ and $c_N= N\E(X\mathbbm{1}_{|X|\leq a_N})$. The sum of eigenvalues satisfies
$$\sum_{i=1}^N \lambda_i=\Tr(M)=\sum_{i=1}^N {M_{ii}}.$$
For $\alpha<1$, Proposition \ref{stable_dist} implies convergence of  $\Tr(M)/a_N$ and $\lim_{N\to \infty} a_N/b_N=0$. This shows $$\lim_{N\to\infty} \sum_{i=1}^N\lambda_i/b_N=0.$$
For $\alpha=1$,
$$\E(X\mathbbm{1}_{|X|\leq a_N})=\int_0^{a_N}\P(|X|> u)\dd u(1+o(1))=O(N^{\epsilon}).$$
for every $\epsilon>0.$
Since $b_N\geq N^{2/\alpha-\epsilon}=N^{2-\epsilon}$ for every $\epsilon>0$, $\lim_{N\to\infty}c_N/b_N=0.$
This and convergence of $(\Tr(M)-c_N)/a_N$ coming from Proposition \ref{stable_dist} implies 
$$\lim_{N\to\infty } \Tr(M)/b_N=0$$
For $\alpha>1$, 
$$\E(X\mathbbm{1}_{|X|\leq a_N})= \int_0^{a_N}\P(|X|> u)\dd u(1+o(1))= O(1).$$
Furthermore, $b_N\geq N^{2/\alpha-\epsilon}$ for every $\epsilon>0$ and $\alpha<2$ implies $\lim_{N\to\infty }N/b_N=0$ and $\lim_{N\to \infty }c_N/b_N=0.$
Therefore, convergence of $(\Tr(M)-c_N)/a_N$ coming from Proposition \ref{stable_dist} and the above implies 
$$\lim_{n\to\infty} \Tr(M)/b_N=0.$$
\end{proof}
\begin{proof}[Proof of Lemma \ref{lem:integral representation}]
We can diagonalize the interaction matrix $M=O^{T}DO$ for an orthogonal matrix O and a diagonal matrix $D=\diag(\lambda_1,..,\lambda_N)$ and apply the change of variable  $x\to O^{-1}x$ then we have
\begin{align*}
    Z_n&=\frac{1}{|S_N|}\int e^{\beta b_N^{-1} <x,Mx>}\dd \Omega=\frac{1}{|S_N|}\int_{S_N }e^{\beta b_N^{-1} \sum_i \lambda_i x_i^2}.
\end{align*}
We will use Laplace transform to calculate above statistics, so we define below statistics. 
\begin{align*}
    J(z)&=\int_{\R^N}e^{\beta N\sum_i\lambda_iy_i^2}e^{-\beta N z\sum_i y_i^2}\dd y, \ \re z>\lambda_1
    \\I(t)&= \int_{\mathbb{S}^{N-1}} e^{t\sum_i \lambda_i x_i^2}\dd\Omega
\end{align*}
$J(z)$ is calculated as
\begin{align*}
    J(z)&=\int_{\R^N}e^{\beta N\sum_i(\lambda_i-z)y_i^2}\dd y
    =(\frac{\pi}{\beta N})^{N/2}\prod_i \frac{1}{\sqrt{z-\lambda_i}}, \re z>\lambda_1
\end{align*}
and
\begin{align*}
    J(z)&=\frac{1}{2(\beta N)^{N/2}}\int_{0}^{\infty}e^{-zt}t^{N/2-1}I(t)\dd t.
\end{align*}
Now we apply Laplace transform and we have
\begin{align*}
    \frac{t^{N/2-1}I(t)}{2(\beta N)^{N/2}}&=\frac{1}{2\pi i}\int_{\gamma -i\infty}^{\gamma + i\infty}e^{zt}J(z)\dd z
    =\left(\frac{\pi}{\beta N}\right)^{N/2}\frac{1}{2\pi i}\int_{\gamma-i\infty}^{\gamma+i\infty}e^{zt}\prod\frac{1}{\sqrt{z-\lambda_i}}\dd z.
\end{align*}
On the other hand,
\begin{align*}
     Z_N&=\frac{1}{S^{N-1}}\int_{S^{N-1}}e^{b_N^{-1}\beta N\sum_i \lambda_i x_i^2}\dd\sigma=\frac{1}{|S^{N-1}|}I(\frac{\beta N}{b_N})
    \\&=\frac{\Gamma(N/2)}{2\pi^{N/2}}\cdot\frac{2\pi^{N/2}b_N^{N/2-1}}{(\beta N)^{N/2-1}}\frac{1}{2\pi i}\int_{\gamma-i\infty}^{\gamma+i\infty}e^{\frac{\beta N z}{b_N}}\prod_{i=1}^N\frac{1}{\sqrt{z-\lambda_i}}\dd z
    \\&=C_N\frac{1}{i}\int_{\gamma-i\infty}^{\gamma+i\infty} e^{\frac{N}{2b_N}G(z)}\dd z
\end{align*}
where
\begin{align*}
    G(z)=2\beta z-\frac{b_N}{N}\sum_{i=1}^N \log(z-\lambda_i), \ C_N=\frac{\Gamma(N/2)b_N^{N/2-1}}{2\pi (N\beta)^{N/2-1}}.
\end{align*}
\end{proof}

\begin{proof}[Proof of Lemma \ref{lem:gamma}]
We first consider the first case, $\lambda_1<\frac{b_N}{2\beta}$. From Lemma \ref{lem:num_eig},
 $$0 \leq \sum_{|\lambda_i|>b_NN^{-\epsilon}}\frac{1}{z-\lambda_i}\leq \#\{ |\lambda_i|>b_N N^{-\epsilon}\} \frac{1}{z-\lambda_1}=O(\frac{N^{3\epsilon}}{z-\lambda_1}).$$
Since we have
    $$z>\lambda_1> b_N N^{-\epsilon/4}$$
with high probability due to Lemma \ref{lem:gap_eig},
    $$ \sum_{|\lambda_i|<b_N N^{-\epsilon}}\frac{1}{z-\lambda_i}= (N-O(N^{3\epsilon}))\frac{1}{z}(1+O(N^{-3/4\epsilon}))=\frac{N(1+O(N^{-3/4\epsilon}))}{z}.$$
    Summing up them and we have
    $$G'(z)=2\beta-\frac{b_N}{N}(O(\frac{N^{3\epsilon}}{z-\lambda_1})+\frac{N(1+O(N^{-3/4\epsilon})}{z})$$
    $$G'(\frac{b_N}{2\beta}-b_NN^{-\epsilon/2})<0, \ G(\frac{b_N}{2\beta}+b_NN^{-\epsilon/2})>0$$
    Here $\lambda_1<\frac{b_N}{2\beta}$ implies $\lambda_1<\frac{b_N}{2\beta}-b_NN^{-\epsilon/4}$ holds with high probability due to Lemma \ref{lem:gap_eig}.
    Since $\lambda_1=\max \{|M_{ij}|(1+O(b_N^{-1/8}))\}$ and $|\lambda_N|=\max \{|M_{ij}|(1+O(b_N^{-1/8}))\}$ due to Lemma \ref{ls_eig},
    $$\gamma=\frac{b_N}{2\beta}(1+O(N^{-\epsilon/2}))>\lambda_1,|\lambda_N|$$
    holds with high probability.
    Hence,
    $$G'(\gamma)=2\beta-\frac{b_N}{N}\sum_{i=1}^N\frac{1}{\gamma-\lambda_i}=2\beta -\frac{b_N}{N \gamma}(N+\sum_{i=1}^N \sum_{k=1}^{\infty}\lambda_i^k/\gamma^k).$$
    \\Let us show that $$X_N=\sum_{i=1}^N \sum_{k=1}^{\infty}\lambda_i^k/\gamma^k$$ converges to non degenerate random variable X.
    
    Lemma \ref{lem:lim_of_Tr} implies 
    $$\lim_{N\to\infty} \sum_{i=1}^N \lambda_i/\gamma= 0.$$
    The remaining terms are
    $$\left|\frac{\sum_{i=1}^N \lambda_i^k}{\gamma^k}\right|=\left|\frac{\sum_{i=1}^N \lambda_i^k}{(b_N/2\beta)^k}(1+O(N^{-\epsilon}))\right|\leq\frac{\Gamma^{k-2}\sum_{i=1}^N \lambda_i^2}{(b_N/2\beta)^k}(1+O(N^{-\epsilon})) $$
    where $\Gamma=\max\{|\lambda_1|,|\lambda_N|\}$.
    \\Since $\Gamma/(b_N/2\beta)$ converges to same random variable with convergence of $\lambda_1/(b_N/2\beta)$ and Lemma \ref{lem:trace_square} implies $\sum_{i=1}^N\lambda_i^2/(b_N/2\beta)^2$ converges to a random variable,
    $$|X_N|\leq \sum_i \lambda_i/\gamma +\frac{\Gamma/(b_N/2\beta)}{1-(\lambda_1/(b_N/2\beta))}\sum_{i=1}^N(\lambda_i/(b_n/2\beta))^2(1+O(N^{-\epsilon}))$$
    implies that its upper bound has convergence random variable.
    \\Using the same method of proving Lemma \ref{lem:T}, this implies $$0=G'(\gamma)= 2\beta -\frac{b_N}{N\gamma}(N+ X_N)$$ and $X_N$ converges to a non degenerate random variable X.
    \\Thus, $$\gamma= \frac{b_N}{2\beta}+\frac{b_N X_N}{2\beta N}$$
    \\Case 2: $\lambda_1>\frac{b_N}{2\beta}$ holds.
    $$\sum_{|\lambda_i|>b_NN^{-\epsilon}} \frac{1}{z-\lambda_i}=\frac{1}{z-\lambda_1}+\frac{O(N^{3\epsilon})}{z-\lambda_2}$$
    Summing up and we obtain
    $$G'(z)=2\beta-\frac{b_N}{N}(\frac{1}{z-\lambda_1}+O(\frac{N^{3\epsilon}}{z-\lambda_2})+\frac{N(1+O(N^{-3/4\epsilon})}{z})$$
    Since $\lambda_1-\lambda_2>b_NN^{-\epsilon}$ holds with high probability due to Lemma \ref{lem:gap_eig},
    $$G'(z)= 2\beta -\frac{b_N}{N}(\frac{1}{z-\lambda_1}+O(\frac{N^{4\epsilon}}{b_N})+\frac{N(1+O(N^{-3/4\epsilon})}{z})$$ for $z>\lambda_1$.
    Hence,
    $$G'(\lambda_1+\frac{1}{2\beta-\frac{b_N}{\lambda_1}}\frac{b_N}{N}+\frac{b_N}{N}N^{-\epsilon/2})>0,G'(\lambda_1+\frac{1}{2\beta-\frac{b_N}{\lambda_1}}\frac{b_N}{N}-\frac{b_N}{N}N^{-\epsilon/2})<0.$$
    This implies $$\gamma=\lambda_1+\frac{1}{2\beta-\frac{b_N}{\lambda_1}}\frac{b_N}{N}+O(\frac{b_N}{N}N^{-\epsilon/2}).$$
    \end{proof}

\section{Statistical property for T}\label{app:stat_T}
We also want to know the property of the statistics T.
Let $T_N=-\sum_i \log (1-\frac{\lambda_i}{\gamma})$
$$\E(-\log (1-\sum_{i=1}^N \lambda_i/\gamma))= \sum_{k=1}^{\infty} \frac{1}{k}\E(\sum_{i=1}^N (\lambda_i/\gamma)^k)$$
In this case is the conditional case and every entry of interaction matrix follows the $Y=X1_{|X|\leq \frac{b_N}{2\beta}}$.
$$\lim_{N\to \infty}\E((\frac{b_N}{2\beta})^{-k}\sum \lambda_i^k)=\begin{cases}
    0 & \textit{k is odd}
    \\ \frac{k}{k-\alpha} &\textit{k is even }
\end{cases}$$
$$\E(\sum_{i=1}^N \lambda_i^k)=\E(Tr M^k)=\E(\sum M_{i_1i_2}...M_{i_ki_1})$$
Since k is finite, if $\{i_1,...,i_k\}$ has t vertices, then the possible choice is $n(n-1)...(n-t+1)$.
\[\E(|Y|)=\int_{0}^{b_N/2\beta}\P(|Y|>u)\dd u= \begin{cases}
    \frac{(2\beta)^{\alpha}}{1-\alpha}\frac{b_N}{2\beta}\frac{2}{N(N+1)}(1+o(1))&  \alpha<1 \\
    O(1) & \textit{otherwise}
\end{cases}\]
For $k\geq 2$
\[
    \E(|Y|^k)=\int_{0}^{b_N/2\beta} ku^{k-1}\P(|Y|>u)\dd u= \frac{(2\beta)^{\alpha}}{k-\alpha}(\frac{b_N}{2\beta})^{k}\frac{2k}{N(N+1)}(1+o(1))
\]
Therefore for the corresponding graph $T$,
\[
\E(\sum_{T} M_{i_1i_2}...M_{i_ki_1})=\prod_{x=1}^{t}(N-x+1)\E(e_1^{t_1}...e_l^{t_l})=\prod_{x=1}^{t}(N-x+1)\prod_{i=1}^{l}\E(Y^{t_i})
\]
For $\alpha<1$,
\[\E(\sum_{T} M_{i_1i_2}...M_{i_ki_1})=(\frac{b_N}{2\beta})^{k}\prod_{x=1}^{t}(N-x+1)\prod_{i=1}^{l}\frac{2(2\beta)^{\alpha}t_i}{(t_i-\alpha)N(N+1)}\]
\[
    \lim_{N\to\infty}\E((\frac{b_N}{2\beta})^{-k}\sum_{T}M_{i_1i_2}...M_{i_ki_1})
\]
it converges not to 0, only if $t=2l$.
$t$ is the number of vertices, and $l$ is the distinct kind of edges in $\{i_1i_2,...,i_ki_1\}$.Since the number of edges are at least $t-1$, $t=2l\geq 2t-2$ and $t\leq 2$. Hence, the only possible case is $l=1, t=2.$ This case is only possible when k is even and when the graph T is $k$ edges between two vertices.
Combining altogether, we have
\[
    \lim_{N\to\infty}\E((\frac{b_N}{2\beta})^{-k}(\sum_{i=1}^N \lambda_i^k))
    =\begin{cases}
       0 & k \textit{ is odd}\\ \frac{2(2\beta)^{\alpha}k}{k-\alpha} & k \textit{ is even} 
    \end{cases}
\]
Therefore, 
\[
    \lim_{N\to\infty}\E(-\sum_{i=1}^N\log(1-\frac{\lambda_i}{b_N/2\beta}))=\lim_{N\to\infty}\E(\sum_{k=1}^{\infty} \frac{1}{k}\sum (\lambda_i/(b_N/2\beta))^k)=\sum_{t=1}^{\infty} \frac{2(2\beta)^{\alpha}}{(2t-\alpha)}
\]
Similarly for $\alpha\geq 1$, let $l'$ be the number of edges appear once in $T$. 
\[\E(\sum_{T} M_{i_1i_2}...M_{i_ki_1})=(\frac{b_N}{2\beta})^{k-l'}\prod_{x=1}^{t}(N-x+1)\prod_{i=1}^{l-l'}\frac{2(2\beta)^{\alpha}t_i}{(t_i-\alpha)N(N+1)}O(1)^{l'}\]
To make the limit
\[
    \lim_{N\to\infty}\E((\frac{b_N}{2\beta})^{-k}\sum_{T}M_{i_1i_2}...M_{i_ki_1})
\]
not to be 0, it should satisfy $l+1\geq t\geq l'(2/\alpha)+2(l-l')$.
Let $l''=l-l'$, then $l'+l''+1\geq (2/\alpha)l'+2l''$.
Since $l''+(\alpha/2-1)\leq 1$, $l''=0 \textit{ or } 1$. If $l''=0$, then  
\[
    \lim_{N\to\infty}\E((\frac{b_N}{2\beta})^{-k}\sum_{T}M_{i_1i_2}...M_{i_ki_1})=(\frac{b_N}{2\beta})^{-k}\prod_{x=1}^{t}(N-x+1)O(1)^k=0
\]
since $t\leq k<\frac{2}{\alpha}k$.
Therefore, the remaining part is $l''=1,l'=0$. This is same with above. Now we have
\[
    \lim_{N\to\infty}\E(-\sum_{i=1}^N\log(1-\frac{\lambda_i}{b_N/2\beta}))=\lim_{N\to\infty}\E(\sum_{k=1}^{\infty} \frac{1}{k}\sum_{i=1}^N (\lambda_i/(b_N/2\beta))^k)=\sum_{t=1}^{\infty} \frac{2(2\beta)^{\alpha}}{(2t-\alpha)}
\]
Due to Lemma \ref{lem:num_eig}, $\lim_{N\to \infty} \sum_{i=1}^N \{\log (1-\frac{\lambda_i}{b_N/2\beta})-\log(1-\frac{\lambda_i}{\gamma})\}=0.$
This implies 
\[\E(T)=\lim_{N\to \infty}\E(\frac{1}{2}T_N)=\sum_{t=1}^{\infty}\frac{(2\beta)^{\alpha}}{(2t-\alpha)}\] 
for any cases.
Now, let us get back to the calculation of the limit of $\E(T^2)$.
To calculate this, we calculate 
\[
    \E(\sum_{i=1}^N \lambda_i^k \sum_{i=1}^N \lambda_i^l)=
    \E(\sum_{T_1\cup T_2}(M_{i_1i_2}...M_{i_ki_1})(M_{j_1j_2}...M_{j_lj_1}))
\]
Let t be the number of elements in $\{i_1i_2,...,i_{k-1}i_k,i_ki_1,j_1j_2,..,j_{l-1}j_l,j_lj_1\}$ and let s be the number of different edges.
For fixed $T_1\cup T_2$, we also have similar equation.
For $\alpha<1$
\[\E((\frac{b_N}{2\beta})^{-k-l}\sum_{T_1\cup T_2} (M_{i_1i_2}...M_{i_ki_1})(M_{j_1j_2}...M_{j_lj_1})))=\prod_{x=1}^{t}(N-x+1)\prod_{i=1}^{s}\frac{2(2\beta)^{\alpha}t_i}{(t_i-\alpha)N(N+1)}\]
It converges only when t=2s, and $s\geq t-2=2s-2$. The possible cases to above equation converges are $(t,s)=(2,1),(4,2)$. More precisely, the possible cases are
\begin{itemize}
    \item   $T_1$ consists of $k$ edges $e_1$, $T_2$ consists of $l$ edges $e_2\neq e_1$.
    \item $T_1$ consists of $k$ edges $e$.
\end{itemize}
$k,l$ are even number for both cases.
Therefore,
\[
    \lim _{N\to\infty}\E((\frac{b_N}{2\beta})^{-k-l}(\sum \lambda_i^k)(\sum \lambda_i^l))=\begin{cases}
    \frac{(2(2\beta)^{\alpha})^2kl}{(k-\alpha)(l-\alpha)}+\frac{2(2\beta)^{\alpha}(k+l)}{k+l-\alpha}& k,l \textit{ are even}
    \\ 0 & \textit{otherwise}
    \end{cases}
\]
For $\alpha\geq1$ case, for the similar reason from the calculation of expectation, it also has the same value as above.

Since
\[
    (\sum \log(1-\lambda_i/(b_N/2\beta))^2=\sum_{k=1}^{\infty}\sum_{l=1}^{\infty}(\frac{b_N}{2\beta})^{-k-l}\frac{1}{kl}\sum \lambda_i^k \sum \lambda_i^l,
\]
\begin{align*}
    \E(T^2)&=\lim_{N\to\infty}\E(\frac{1}{4}T_N^2)
    =\lim_{N\to \infty} (\sum \log(1-\lambda_i/(b_N/2\beta))^2\\&=
    \sum_{k=1}^{\infty}\sum_{l=1}^{\infty} \frac{((2\beta)^\alpha)^2 4kl}{2k(2k-\alpha)2l(2l-\alpha)}+ \frac{(2\beta)^{\alpha}(k+l)}{4kl(2k+2l-\alpha)}.
\end{align*}
Moreover, the variance is
\[
    V(T)=\lim_{N\to\infty}V(\frac{1}{2}T_N)= \sum_{k=1}^{\infty}\sum_{l=1}^{\infty}\frac{(2\beta)^{\alpha}(k+l)}{4kl(2k+2l-\alpha)}=\sum_{n=1}^{\infty}\frac{(2\beta)^{\alpha}}{2(2n+2-\alpha)}\sum_{i=1}^n \frac{1}{i}.
\]

\subsection*{Acknowledgement}
 The first author thanks to Ji Hyung Jung for helpful discussions. The work was partially supported by National Research Foundation of Korea under grant number NRF-2019R1A5A1028324 and NRF-2023R1A2C1005843.
\subsection*{Conflict of interest}
The authors have no relevant financial or non-financial interests to disclose.
\subsection*{Data Availability}
No datasets were generated or analysed during the current study.

\end{document}